\documentclass[12pt]{article}

\usepackage{latexsym}
\usepackage{a4wide}
\usepackage{amscd}
\usepackage{graphicx}
\usepackage{amsmath}
\usepackage{amssymb}
\usepackage{soul,xcolor}
\usepackage{esint}
\usepackage{mathrsfs}
\usepackage{amsthm}
\usepackage{bbm}
\usepackage{color}
\usepackage{accents}
\usepackage{enumerate}
\usepackage{hyperref}
\input xy
\xyoption{all}

\newcommand{\Z}{\mathbb{Z}}                     
\newcommand{\R}{\mathbb{R}}                     

\newtheorem{mainthm}{\sc Theorem}           
\newtheorem{thm}{\sc Theorem}[section]               
\newtheorem*{thm*}{\sc Theorem}               
\newtheorem*{cor*}{\sc Corollary}        
\newtheorem{lem}[thm]{\sc Lemma}            
\newtheorem{rem}[thm]{\sc Remark}           

\title{Sharp systolic inequalities for Riemannian and Finsler spheres of revolution} 

\author{Alberto Abbondandolo, Barney Bramham, \\ Umberto L.~Hryniewicz, Pedro A.~S.~Salom\~ao}

\begin{document}

\maketitle

\begin{abstract}
We prove that the systolic ratio of a sphere of revolution $S$ does not exceed $\pi$ and equals $\pi$ if and only if $S$ is Zoll. More generally, we consider the rotationally symmetric Finsler metrics on a sphere of revolution which are defined by shifting the tangent unit circles by a Killing vector field. We prove that in this class of metrics the systolic ratio does not exceed $\pi$ and equals $\pi$ if and only if the metric is Riemannian and Zoll.
\end{abstract}

\section*{Introduction}

The systolic ratio of a Riemannian (two-dimensional) sphere $S$ is the positive number
\[
\rho_{\mathrm{sys}}(S) := \frac{\ell_{\min}(S)^2}{\mathrm{area}(S)},
\]
where $\ell_{\min}(S)$ denotes the length of the shortest non-constant closed geodesic on $S$ and $\mathrm{area}(S)$ is the Riemannian area of $S$. This number is clearly invariant by isometries and rescaling. Moreover, the systolic ratio is a dynamical invariant: If the geodesic flows of two Riemannian spheres $S$ and $S'$ are smoothly conjugate, then $\rho_{\mathrm{sys}}(S)=\rho_{\mathrm{sys}}(S')$. Here we recall that the geodesic flow of $S$ is the flow on the unit tangent bundle $T^1 S$ whose time $t$-maps send the unit tangent vector $u$ into the velocity vector $\dot{\gamma}(t)$ of the geodesic $\gamma$ such that $\dot\gamma(0)=u$.  Actually, $\ell_{\min}$ is a dynamical invariant, just because this number is the shortest period of a closed orbit of the geodesic flow, and the Riemannian area of a closed surface (or more generally the volume of a closed Riemannian manifold) is a $C^1$-conjugacy invariant of the geodesic flow (see \cite[Proposition 1.2]{ck94}). 

Clearly, the systolic ratio of a Riemannian sphere can be arbitrarily small. In \cite{cro88}, Croke proved that $\rho_{\mathrm{sys}}$ is bounded from above on the space of all Riemannian spheres. The value of the supremum  is not known, but it is conjectured to be $2\sqrt{3}= 3.46\dots$. This number is the systolic ratio of the sphere with three conical singularities which is obtained by gluing two flat equilateral triangles by their sides and is known as the Calabi-Croke sphere, see \cite{cro88,ck03}. The best known upper bound for $\rho_{\mathrm{sys}}$ on Riemannian spheres is 32 and is due to Rotmann, see \cite{rot06}.

The value of the systolic ratio on a round sphere (of any radius) is $\pi$. Actually, all Zoll spheres have systolic ratio $\pi$. We recall that a Riemannian sphere is called Zoll if all its geodesics are closed and have the same length. This is actually equivalent to asking that all the geodesics are closed, see \cite{gg81}. Zoll spheres are named after Otto Zoll, who exhibited the first example of an analytic surface of revolution with this property which is not isometric to a round sphere, see \cite{zol03}. Zoll spheres form a huge infinite dimensional space inside Riemannian spheres, whose structure is only partially understood, see \cite{gui76}. The fact that all Zoll spheres have systolic ratio $\pi$ has been known for a long time, see \cite{wei74}, but can also be deduced from the fact that the geodesic flow of a Zoll sphere is smoothly conjugate to the one of the round sphere whose geodesics have the same length, see \cite[Theorem B.1]{abhs17}.

Zoll spheres are local maximizers of the systolic ratio in the $C^2$-topology of Riemannian metrics: There exists a $C^2$ neighborhood $\mathcal{U}$ of the space of Zoll metrics on $S^2$ such that any sphere $S$ whose metric is in $\mathcal{U}$ satisfies $\rho_{\mathrm{sys}}(S)\leq \pi$, with the equality holding if and only if $S$ is Zoll. This local maximality property was conjectured for the round metric by Babenko and Balacheff, see \cite{bal06}, and then proved by the authors in \cite{abhs17} for positively curved spheres satisfying a pinching condition and in \cite[Corollary 4]{abhs18} for arbitrary spheres. Actually, the local maximality of the systolic ratio proved in  the latter paper is with respect to the $C^3$ topology of Finsler metrics. In the case of Riemannian metrics, or more generally reversible Finsler metrics, the proof from \cite{abhs18} can be modified in order to to guarantee local maximality in the $C^2$ topology. Details about this will appear elsewhere.

In this paper, we focus our attention on spheres of revolution. By a sphere of revolution we mean here a smooth surface $S$ in $\R^3$ which is diffeomorphic to a sphere and is invariant with respect to the rotations around the $z$-axis. Such a surface is uniquely determined by its intersection with any plane containing the $z$-axis, which is necessarily a smooth embedded closed curve, symmetric with respect to the $z$-axis. See Figure~1. The first result of this paper is the following:

\begin{figure}
\begin{center}
\includegraphics[scale=0.75]{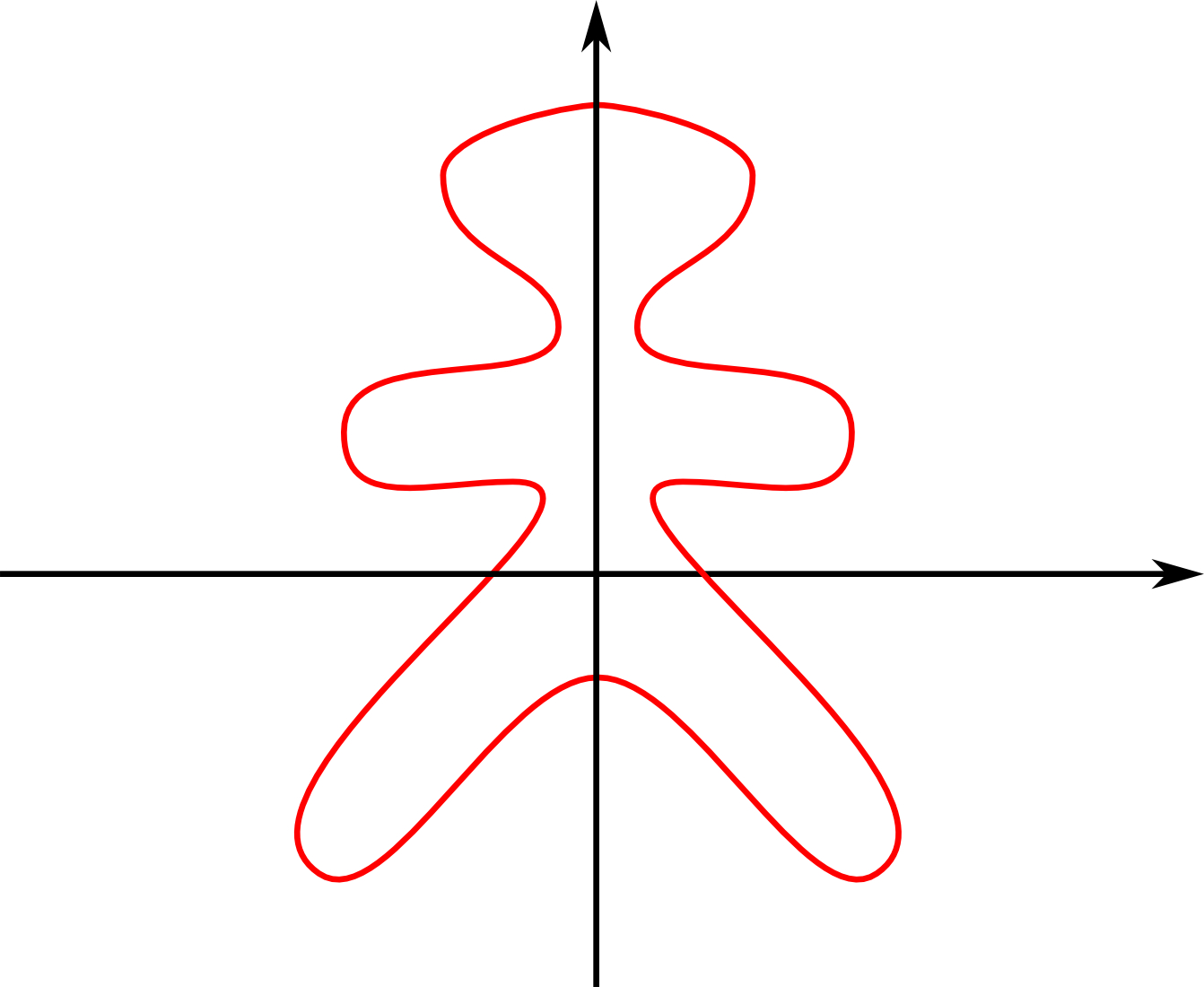}
\caption{A symmetric curve.}
\end{center}
\end{figure}

\begin{mainthm}
\label{mainthm1}
The systolic ratio of a sphere of revolution $S$ does not exceed $\pi$. It equals $\pi$ if and only if $S$ is Zoll.
\end{mainthm}

What makes this result interesting and its proof non-trivial is the fact that in the space of spheres of revolution there is a huge subset of Zoll spheres, and hence of spheres for which the above inequality is an equality. A simple way to exhibit these examples is described in Zoll's paper \cite{zol03}, building on previous work by Darboux: One considers an embedded curve $\sigma$ in the $(r,z)$-plane which is symmetric with respect to the $z$-axis and is given by the union of the graph of two even functions
\[
z_-:[-R,R] \rightarrow \R \qquad \mbox{and} \qquad  z_+:[-R,R] \rightarrow \R
\]
such that $z_{\pm}(\pm R)=0$ and $z_-<z_+$ on $(-R,R)$, see Figure~2. If the functions $z_-$ and $z_+$ satisfy the condition
\begin{equation}
\label{darboux}
\sqrt{1+z_-'(r)^2} + \sqrt{1+z_+'(r)^2} = \frac{2R}{\sqrt{R^2-r^2}} \qquad \forall r\in (-R,R),
\end{equation}
then the resulting sphere of revolution $S$ is Zoll. Notice that it is easy to find pairs of functions $z_-,z_+$ satisfying the above equation and producing a smooth embedded curve $\sigma$, and hence a smooth sphere $S$: One can start from an arbitrary even function $z_+:[-R,R] \rightarrow \R$ which is smooth on $(-R,R)$, agrees with the function
\[
r\mapsto \sqrt{R^2 - r^2}
\]
in a neighborhood of $R$ and $-R$ and satisfies the inequality
\[
\frac{2R}{\sqrt{R^2-r^2}} - \sqrt{1+z_+'(r)^2} > 1 \qquad \qquad \forall r\in (-R,R)\setminus \{0\}.
\] 
Then the pointwise equation (\ref{darboux}) uniquely determines a non-negative even function $r\mapsto z_-'(r)^2$ on $(-R,R)$, vanishing only at $r=0$, and by integration we find a unique function $z_-$ such that $z_+$ and $z_-$ satisfy (\ref{darboux}). The resulting curve $\sigma$ is smoothly embedded. See \cite{bes78} for more results on Zoll spheres.

\begin{figure}
\begin{center}
\includegraphics[scale=0.75]{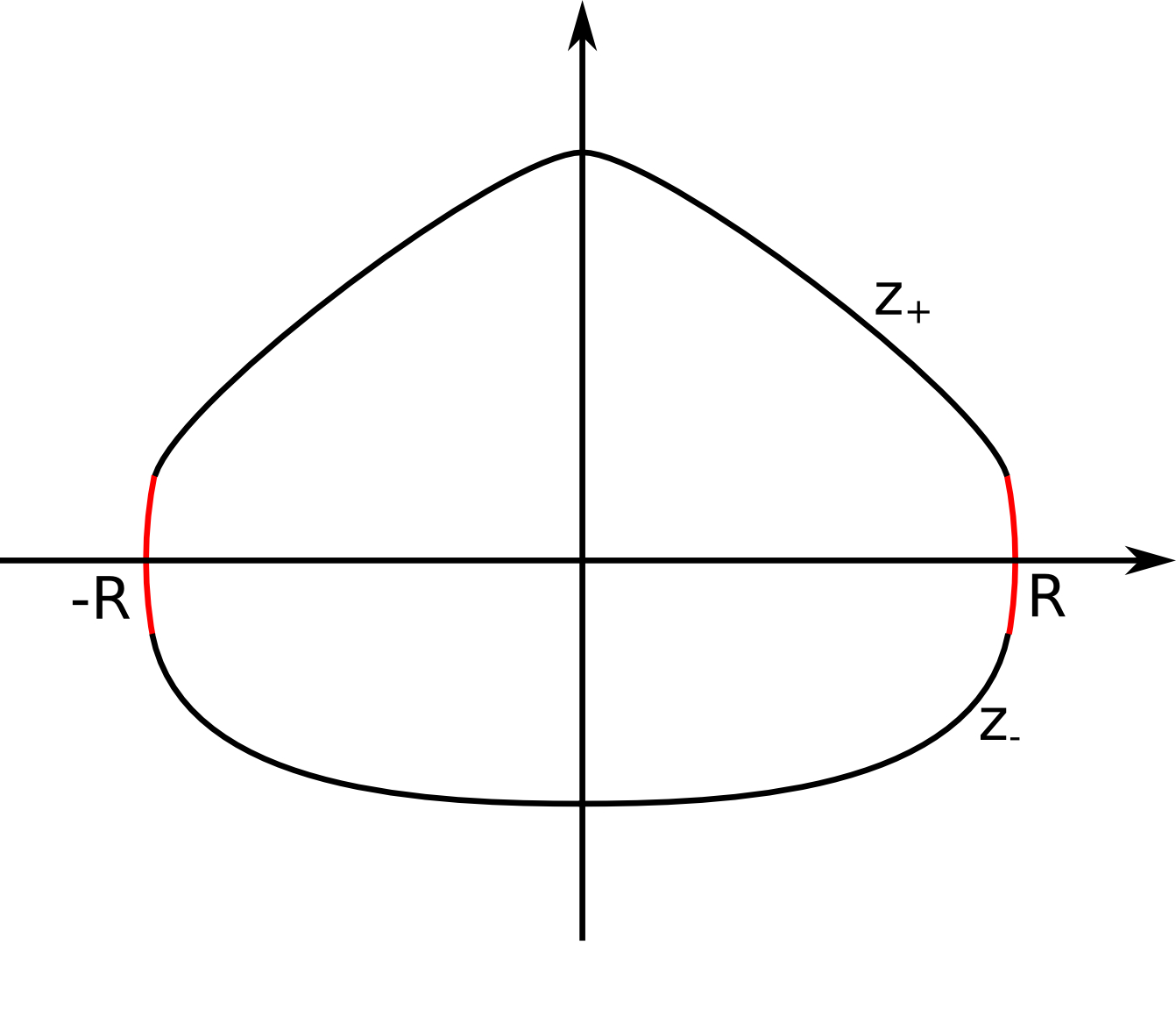}
\caption{Graphs of $z_+$ and $z_-$.}
\end{center}
\end{figure}

In our second theorem, we wish to extend the above systolic inequality to a class of Finsler metrics on $S^2$ arising in Zermelo's navigation problem, see \cite{zer31}. These Finsler metrics are constructed by starting with a surface of revolution $S$ in $\R^3$ and by translating all unit tangent circles
\[
T_p^1 S := \{v\in T_p S \mid \|v\|=1\}
\]
by the Killing vector field
\[
W_a (x,y,z) :=  a \left( x \frac{\partial}{\partial y} - y \frac{\partial}{\partial x} \right) = a \frac{\partial}{\partial \theta},
\]
which is referred to as ``wind''. Here $a$ is a real number whose absolute value is smaller than $1/r_{\max}$, the inverse of the maximal distance of points of $S$ from the $z$-axis, and $\theta$ is the angular coordinate in the plane $x,y$. The condition on $a$ guarantees that the translated circles $W_a(p) + T_p^1 S$ still bound open disks containing the origin of $T_p S$, and hence they are the unit spheres of a Finsler metric $G_a$ on $S$. This Finsler metric is non-reversible when $a\neq 0$. 

There are two possible notions of area which are associated to a Finsler metric $G$ on the surface $S$. The first one is known as Busemann-Hausdorff area and is obtained by integrating over $S$ the following area form $\rho_G$: let $\rho$ be any non-vanishing area form and set
\[
\rho_G(p) := \frac{\pi}{|B_G(p)|_{\rho}}  \rho(p),
\]
where $B_G(p)\subset T_p S$ denotes the unit ball of the Finsler metric $G$ at the point $p\in S$ and $|\cdot|_{\rho}$ denotes the Lebesgue measure on $T_p S^2$ normalized so to give area 1 to the parallelogram spanned by two vectors $v,w\in T_p S$ with $\rho(p)[v,w]=1$. The  Busemann-Hausdorff area of $(S,G)$ is the positive number
\[
\mathrm{area}_{\rm BH}(S,G) := \int_{S} \rho_G,
\]
where $S$ has the orientation induced by $\rho$. The normalization constant in the definition of $\rho_G$ is chosen in such a way that $\mathrm{area}_{BH}(S,G)$ coincides with the Riemannian area of $(S,G)$ when the metric $G$ is Riemannian. The second possibility is to consider the Holmes-Thompson area of $(S,G)$, which is defined as
\[
\mathrm{area}_{\rm HT}(S,G) := \frac{1}{2\pi} \mathrm{vol} \Bigl( \bigcup_{p\in S} B_G^*(p) \Bigr),
\]
where $B_G^*(p)\subset T_p^* S$ denotes the polar set of $B_G(p)$ and $\mathrm{vol}$ denotes the volume on $T^* S$ given by volume form $\omega\wedge \omega$ which is induced by the standard symplectic form $\omega$ of the cotangent bundle $T^* S$. Equivalently, $\mathrm{area}_{HT}(S,G)$ is the integral over $S$ of the area form
\[
\rho_G^*(p) := \frac{|B_G^*(p)|_{\rho}^*}{\pi}  \rho(p),
\]
where $|\cdot|_{\rho}^*$ is the Lebesgue measure on $T_p^* S$ normalized so to give area 1 to the parallelogram spanned by two covectors $\xi,\eta\in T_p^* S$ forming a base which is dual to a base $v,w$ of $T_p S$ with $\rho(p)[v,w]=1$. Again, normalization constants are chosen so that $\mathrm{area}_{HT}(S,G)$ agrees with the Riemannian area when the metric $G$ is Riemannian. See \cite{she01} for more information on these two different ways of measuring area in Finsler geometry.

In systolic questions, the Holmes-Thompson area is probably more relevant, because of its symplectic nature. Still, here we consider both possibilities for defining the systolic ratio of $(S,G)$:
\[
\rho^{\rm BH}_{\rm sys}(S,G) := \frac{\ell_{\min}(S,G)^2}{\mathrm{area}_{\rm BH}(S,G)} \qquad \mbox{and} \qquad \rho^{\rm HT}_{\rm sys}(S,G) := \frac{\ell_{\min}(S,G)^2}{\mathrm{area}_{\rm HT}(S,G)},
\]
where $\ell_{\min}(S,G)$ denotes the length of the shortest non-constant closed geodesic on $(S,G)$. 

In the special case of the metric $G_a$ on the sphere of revolution $S\subset \R^3$ which is described above, we clearly have that the Busemann-Hausdorff area coincides with the Riemannian area of $S$,
\[
\mathrm{area}_{\rm BH}(S,G_a) = \mathrm{area} (S),
\]
because the Finsler unit balls $B_{G_a}(p)$ are just translations of the Riemannian ones. On the other hand, the Holmes-Thompson area of $(S,G_a)$ is surely larger than the above number,
\[
\mathrm{area}_{\rm HT}(S,G_a) > \mathrm{area} (S) \qquad \forall a\in \bigl(-1/r_{\max},1/r_{\max}\bigr) \setminus \{0\}.
\]
This follows from the fact that the area of the polar set of the translation $K+v$ of a centrally symmetric convex planar body $K$ is strictly larger than the area of the polar of $K$ when $v$ is a non-vanishing vector, see e.g. \cite{sch14}. 
In our case, $K$ is a disk and this fact can be verified by an elementary argument: The polar set of a unit disk in $\R^2$ centered at the point $(a,0)$ with $|a|<1$ is the ellipse 
\[
\{(p_1,p_2)\in \R^2 \mid (1-a^2)p_1^2 + p_2^2 + 2 a p_1 \leq 1\},
\]
which has area $\pi (1-a^2)^{-3/2}$.

After these preliminaries, we can state the following result, which generalizes Theorem~\ref{mainthm1}:

\begin{mainthm}
\label{mainthm2}
Let $S\subset \R^3$ be a surface of revolution and let $a$ be a real number whose absolute value is smaller than $1/r_{\max}$, where $r_{\max}$ denotes the maximal distance of a point in $S$ from the $z$-axis. Then 
\[
\rho^{\rm HT}_{\rm sys}(S,G_a) \leq \rho^{\rm BH}_{\rm sys}(S,G_a) \leq \pi.
\]
The first inequality is an equality if and only if $a=0$. The second one is an equality if and only if $a=0$ and $S$ is Zoll.
\end{mainthm}

The proof of Theorem \ref{mainthm2} is based on similar ideas 
to that of Theorem \ref{mainthm1}
and requires some further computations which are contained in the last three sections of the paper.

\begin{figure}
\begin{center}
\includegraphics[scale=0.75]{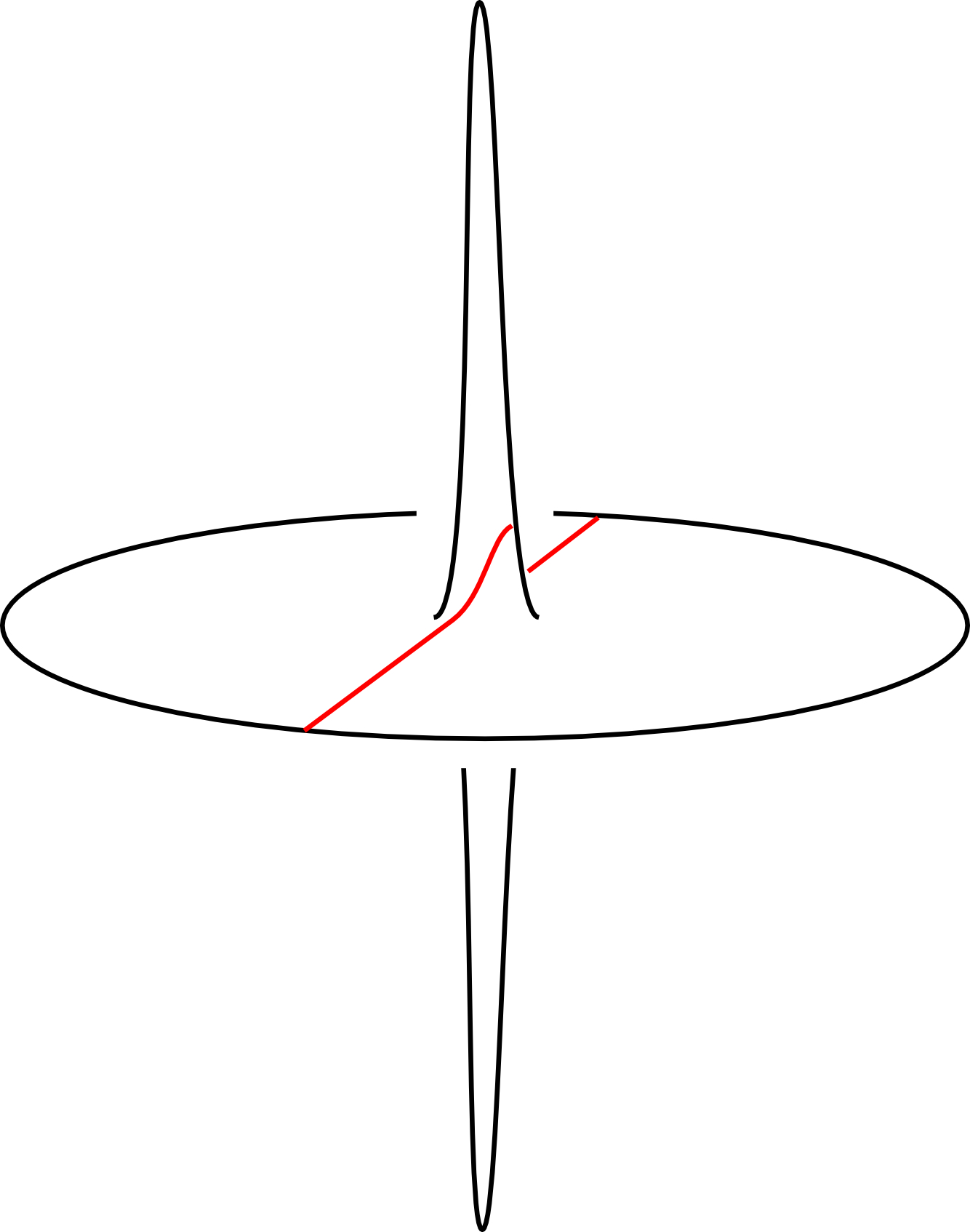}
\caption{A flat-spiked two sphere of revolution: neither a meridian nor an equator has length equal to $\ell_{\rm min}$.}
\end{center}
\end{figure}

Recall that on a sphere of revolution $S$ all meridians (intersections of $S$ with planes containing the $z$-axis) are closed geodesics, and so are the equators (horizontal circles whose distance from the $z$-axis is stationary). Proving an upper bound for the systolic ratio requires showing the existence of a closed geodesic with a certain upper bound on the length. Notice that we cannot hope to get the sharp bound $\rho_{\mathrm{sys}}(S)\leq \pi$ for a sphere of revolution $S$ just by looking at meridians and equators. Indeed, consider a sphere of revolution looking like a flat disk of radius $R$ with a long spike carrying very small area around the axis of revolution as in Figure~3. Its area is close to $2 \pi R^2$. The unique equator has length $2\pi R$, and the quotient between the square of its length and the area is close to
\[
\frac{(2\pi R)^2}{2 \pi R^2} = 2 \pi.
\]
So this closed geodesic misses the bound we wish to prove by a factor of 2.  The meridians of such a sphere can be arbitrarily long, so they also fail to give us the desired upper bound $\pi$ on the systolic ratio. In this example, closed geodesics of minimal length are neither equators nor meridians. 

The bound $\rho_{\mathrm{sys}}(S)< 2\pi$ for any sphere of revolution $S$ is easily proven by considering the equator of maximal length and by bounding from below the area of $S$ by twice the area of the disk bounded by this equator. Getting the sharp bound $\pi$ requires a more sophisticated argument,
which, although ultimately based on the minimization of a function of one variable,
is based on ideas from three-dimensional Reeb dynamics. 

\medskip

We conclude this introduction by sketching 
the proof of Theorem \ref{mainthm1}.
Consider an equator of $S$ having minimal length among all equators, denote by $L$ its length and fix one of its two orientations. Denote by $A$ the open Birkhoff annulus determined by this equator: $A\subset T^1 S$ is the set of unit tangent vectors which are based at the equator and form an angle $\beta\in (0,\pi)$ with the positively oriented tangent vectors to the equator. The set $A$ is a smooth open surface in the three-dimensional manifold $T^1 S$ and is transverse to the geodesic flow. The fact that the length of the equator is minimal guarantees that the forward and backward orbits of any $u\in A$ by the geodesic flow intersect $A$ again. This fact will be deduced by the conservation of Clairaut's integral on $T^1 S$, which takes the form
\[
K(u) := r(u) \cos \beta(u),
\]
where $r(u)$ denotes the distance of the based point of $u$ from the $z$-axis and $\beta(u)$ the angle which $u$ makes with the positive direction of the parallel through its base point. Therefore, the geodesic flow $\phi^t$ on $T^1 S$ has a first return time to $A$
\[
\tau: A \rightarrow (0,+\infty), \qquad \tau(u) := \min \{t\in (0,+\infty) \mid \phi^t(u)\in A\},
\]
and a first return map
\[
\varphi: A \rightarrow A, \qquad \varphi(u) := \phi^{\tau(u)}(u).
\]
The function $\tau$ is smooth on $A$, and the map $\varphi$ is a smooth diffeomorphism which preserves the area form 
\[
\omega := \sin \beta \, d\xi\wedge d\beta = d\xi \wedge d\eta, 
\]
where $\xi\in \R/L \Z$ denotes the arc-length parameter on the equator and $\eta:=-\cos \beta\in (-1,1)$. Using the rotational symmetry and the conservation of Clairaut's integral, we shall prove that in the coordinates $(\xi,\eta)$ the first return time and first return map take the simple form
\begin{equation}
\label{tauphi}
\tau(\xi,\eta) = F(\eta) - \eta F'(\eta), 
\qquad \varphi(\xi,\eta) = (\xi+F'(\eta),\eta),
\end{equation}
where $F:(-1,1)\rightarrow \R$ is an even smooth function whose value at 0 is the length of the meridians. We will refer to $F$ as the generating function of $\varphi$. Any critical point $\eta_0$ of the function $F$ produces a family $(\xi,\eta_0)$, $\xi\in \R/L \Z$, of fixed points of the map $\varphi$, and hence of closed geodesics of length $\tau(\xi,\eta_0)=F(\eta_0)$. Geometric considerations force $F$ to satisfy the lower bound
\begin{equation}
\label{vineq}
F(\eta)\geq L |\eta| \qquad \forall \eta\in (-1,1).
\end{equation}
The generating function $F$ encloses all the relevant information about the dynamics of the geodesic flow on the invariant region of $T^1 S$ which is spanned by the evolution of $A$, and also some information on the complement of this region, as we are going to show.

The main part of the proof consists in understanding the boundary behaviour of the generating function $F$ at $\eta=\pm 1$ and in computing the area of $S$ in terms of  $F$. As is well known, the geodesic flow is the Reeb flow of a contact form $\alpha$ on $T^1 S$, which is called the Hilbert form, and the area of $S$ coincides with the contact volume of $T^1 S$ divided by $2\pi$. Therefore, we may as well compute the contact volume of $T^1 S$. The evolution of the Birkhoff annulus $A$ by the geodesic flow spans an open invariant subset $\tilde{A}\subset T^1 S$ whose contact volume is easily seen to be:
\[
\mathrm{vol}(\tilde{A}) = \int_{A} \tau \, \omega = L \int_{-1}^1 \bigl( F(\eta) - \eta F'(\eta) \bigr) \, d\eta.
\]
In order to simplify this formula, we would like to perform an integration by parts, and this requires showing that the generating function $F:(-1,1)\rightarrow \R$ extends continuously to the closed interval $[-1,1]$ and computing $F(-1)=F(1)$. Moreover, we would like to have a formula for the contact volume of the complement of $\tilde{A}$. It turns out that these two problems are closely related. In order to explain this relation, denote by $M$ the length of the meridians, consider the arclength parameter $s\in (0,M/2)$ on the portion of a meridian joining the two poles, and denote by $r(s)$ the distance of the corresponding point from the $z$-axis.  With this notation, the Clairaut integral takes the form
\[
K(\beta,s) = r(s) \cos \beta.
\]
We shall prove that $F$ extends continuously to $[-1,1]$ by setting
\begin{equation}
\label{bordo}
F(-1)=F(1) := L + \int_{\Gamma} \cos \beta\, d\beta\wedge ds,
\end{equation}
where $\Gamma$ is the following compact subset of the open rectangle $Q:= (-\pi/2, \pi/2) \times (0,M/2)$:
\[
\Gamma := \{ (\beta,s) \in Q \mid 2\pi K(\beta,s) \geq L \}.
\]
The geometric interpretation of the set $\Gamma$ is that this set, together with the corresponding subset of the rectangle  $Q':= (\pi/2,3\pi/2) \times (0,M/2)$, gives us exactly those unit tangent vectors which are not reached by the evolution of the unit tangent vectors in the Birkhoff annulus $A$. Because of this, it should not be surprising that the set $\Gamma$ appears also in the formula for the contact volume of the complement of $\tilde{A}$, which indeed reads
\[
\mathrm{vol}(T^1 S \setminus \tilde{A}) = 4 \pi \int_{\Gamma} r(s) d\beta \wedge ds.
\]
Putting the above formulas together, we obtain the following identity for the contact volume of $T^1 S$:
\[
\mathrm{vol}(T^1S) = 4 L \int_0^1 F(\eta)\, d\eta - 2 L^2 + \int_{\Gamma} (4\pi r(s) - 2 L \cos \beta)\, d\beta\wedge ds.
\]
From this identity it will be easy to derive the lower bound 
\begin{equation}
\label{stvol}
\mathrm{vol}(T^1 S) \geq 4 L \int_0^1 F(\eta)\, d\eta- 2 L^2.
\end{equation}
Moreover, formula (\ref{bordo}) will allow us to show that $S$ is Zoll if and only if the generating function $F$ is constant. 

Since Zoll spheres have systolic ratio $\pi$, in order to prove Theorem \ref{mainthm1} we need to show that if $S$ is not Zoll (and hence, by what we have said above, $F$ is not
constant) then the systolic ratio of $S$ is less than $\pi$. We may assume that $L^2$ is at least equal to $\pi$ times the area of $S$, that is, to half of the contact volume of $T^1 S$, because otherwise the minimal equator is short enough to give us the required upper bound on the systolic ratio. Together with (\ref{stvol}), we deduce that
\[
\int_0^1 F(\eta)\, d\eta \leq L.
\]
The fact that $F$ is not constant and the lower bound $F(1)\geq L$ imply that the even function $F$ achieves its minimum at an interior point $\eta_0\in (-1,1)$ and that $\mu:=F(\eta_0)<L$. This is a critical point of $F$ which gives us a closed geodesic $\gamma$ of length $\mu$. Using also the inequality (\ref{vineq}) we easily obtain the lower bound
\[
\int_0^1 F(\eta)\, d\eta > \mu + \frac{1}{2} \frac{(L-\mu)^2}{L},
\]
which together with (\ref{stvol}) gives us
\[
2\pi \, \mathrm{area}(S) = \mathrm{vol}(T^1 S) > 2\mu^2 = 2 \, \ell(\gamma)^2.
\]
The existence of a closed geodesic whose square of the length is less than $\pi \, \mathrm{area}(S)$  shows that $\rho_{\mathrm{sys}}(S)<\pi$, as we wished to prove. This concludes the sketch of the proof of Theorem~\ref{mainthm1}.

In the Finsler setting of Theorem \ref{mainthm2}, the same approach works but the above formulas require some modification. Indeed, the first return time $\tau$ to the Birkhoff annulus determined by an equator of minimal radius is still given by the first identity in (\ref{tauphi}), but the first return map is affected by the wind $W_a$ and is given by
\[
\varphi(\xi,\eta) = (\xi + F_a'(\eta),\eta),
\]
where $F_a$ is the function
\[
F_a(\eta) := \left( 1- \frac{aL}{2\pi} \eta \right) F(\eta) + \frac{aL}{\pi} \int_0^{\eta} F(\zeta)\, d\zeta.
\]
In particular, $F_a$ is not even when $a\neq 0$. The proof of Theorem \ref{mainthm2} consists in finding a critical point of $F_a$ with small value of $\tau$, under the assumption that the Finsler length of the equator of minimal radius parametrized in the direction of the wind is not small enough to guarantee that the Busemann-Hausdorff systolic ratio of $(S,G_a)$ is smaller than~$\pi$. 

\paragraph{\bf Acknowledgments.} We would like to thank Juan Carlos \'Alvarez Paiva for suggesting to us the problem of maximizing the systolic ratio on spheres of revolution. We are also grateful to Gerhard Knieper for precious bibliographical suggestions and to Erasmo Caponio for a discussion which lead us to the current formulation of Theorem \ref{mainthm2}. The research of A.\ Abbondandolo and B.\ Bramham is supported by the SFB/TRR 191 ``Symplectic Structures in Geometry, Algebra and Dynamics'', funded by the Deutsche Forschungsgemeinschaft. U.\ Hryniewicz was supported by CNPq grant 309966/2016-7 and by the Humboldt Foundation, and acknowledges the generous hospitality of the Ruhr-Universit\"at Bochum. P.\ A.\ S.\ Salom\~{a}o is supported by the FAPESP grant 2011/16265-8 and the CNPq grant 306106/2016-7.

\section{The geodesic flow on a sphere of revolution}

\label{sphererev}
Let $S$ be a  {\em sphere of revolution}, that is, a smooth surface in $\R^3$ which is diffeomorphic to a sphere and is invariant with respect to the rotations around the $z$-axis. Such a surface is uniquely determined by its intersection with the $(x,z)$-plane, which is necessarily a smooth embedded closed curve, symmetric with respect to the $z$-axis. The symmetry and the embedding condition imply that this curve meets the $z$-axis orthogonally at exactly two points, which we call {\em poles}. The pole with smaller value of $z$ is called {\em south pole} and the other one {\em north pole}. We denote by $M$ the length of this closed curve and we parametrize it by arc length by the map 
\[
\sigma: \R/M \Z \rightarrow \R^2 
\]
in such a way that the first component of $\sigma(s)$ is positive for $s\in (0,M/2)$, $\sigma(0)$ corresponds to the south pole and $\sigma(M/2)$ to the north pole. If we denote by 
\[
\sigma(s)=(r(s),z(s)), \qquad s\in \R/M \Z,
\] 
the components of the curve $\sigma$, we obtain that the sphere of revolution $S$ is the set
\[
S =\bigl\{ (r(s)\cos \theta,r(s)\sin \theta,z(s))\in \R^3 \mid  \theta\in \R/2\pi \Z, s\in [0,M/2]\bigr\},
\]
and that the south and north poles are the points $p_S=(0,0,z(0))$ and $p_N=(0,0,z(M/2))$. Notice that the restriction of the function $r$ to the interval $[0,M/2]$ satisfies:
\[
r(0)=r(M/2)=0, \quad r'(0)=1, \quad r'(M/2)=-1, \quad r(s)>0 \;\; \forall s\in (0,M/2).
\]
We consider $S$ as a Riemannian manifold with the metric induced by the Euclidean metric of $\R^3$. The pullback of this metric with respect to the diffeomorphism
\[
\Phi: \R/2\pi \Z \times (0,M/2) \rightarrow S\setminus \{p_S,p_N\}, \qquad \Phi(\theta,s) = \bigl( r(s) \cos \theta, r(s)\sin \theta, z(s) \bigr),
\]
has the form
\begin{equation}
\label{metric}
r(s)^2 d\theta^2 + ds^2.
\end{equation}
We recall some well known facts about geodesics on $S$. We shall always parametrize  geodesics by arc length.
All {\em meridians}, that is unit speed curves parametrizing the intersection of $S$ with a plane containing the $z$-axis, are closed geodesics of length $M$. The {\em parallels} are the arc length reparametrizations of the circles
\[
P_s:= \bigl\{ (r(s)\cos \theta,r(s)\sin \theta,z(s)) \mid \theta\in \R/2\pi \Z \bigr\},
\]
where $s\in (0,M/2)$. We orient the parallels counterclockwise, that is, by declaring the tangent vectors $(-r(s)\sin \theta,r(s) \cos \theta,0)$ to be positive.
The curve $P_s$ is a (necessarily closed) geodesic if and only if $s$ is a critical point of the function $r$. These closed geodesics are called {\em equators}. For $s=0$ and $s=M/2$ the parallels degenerate to the south and north poles:
\[
P_0 := \{p_S\}, \qquad P_{M/2} := \{p_N\}.
\]
We denote by $T^1 S$ the unit tangent bundle of $S$ and by $\phi^t : T^1 S \rightarrow T^1 S$ the geodesic flow. The flow $\phi^t$ is the Reeb flow of the {\em Hilbert contact form} $\alpha$ on $T^1 S$, that is, the contact form which is obtained by restricting the canonical Liouville form of the cotangent bundle of $S$ to the unit cotangent bundle and then pulling it back to $T^1 S$ by the bundle isomorphism which is induced by the metric. 

If $u$ is a unit tangent vector to $S$ at a point $p = \Phi(\theta,s) \in S$ different from the two poles, we denote by $\beta=\beta(u)$ the angle which $u$ makes with the positive direction of the parallel $P_s$ passing through $p$. By taking (\ref{metric}) into account, we see that the unit tangent bundle of $S$ minus the two circles given by the unit tangent vectors at the two poles is the image of the diffeomorphism
\begin{equation}
\label{Psi}
\begin{split}
\Psi: \R/2\pi \Z  \times \R/2\pi \Z \times (0,M/2) \rightarrow  T^1 S \setminus ( T^1_{p_S} S \cup T^1_{p_N} S ), \\
\Psi(\theta,\beta,s) = \left( \Phi(\theta,s), \frac{1}{r} \cos \beta \frac{\partial \Phi}{\partial \theta} (\theta,s) + \sin \beta \frac{\partial \Phi}{\partial s} (\theta,s) \right).
\end{split}
\end{equation}
We will regularly use the above diffeomorphism as a coordinate system to represent unit tangent vectors not based at the two poles.
In the above coordinate system, the Hilbert contact form $\alpha$ is easily seen to be
\begin{equation}
\label{contact}
\alpha(\theta,\beta,s) = r(s) \cos \beta \, d\theta + \sin \beta \, ds.
\end{equation}
Therefore, the contact volume form of $T^1 S$ in this coordinate system is
\begin{equation}
\label{volume-form}
\alpha\wedge d\alpha (\theta,\beta,s) = r(s) \, d\theta\wedge d\beta \wedge ds.
\end{equation}
We recall that the contact volume of $T^1 S$ coincides with $2\pi$ times the Riemannian area of $S$. In the case of our sphere of revolution, this general fact produces the identity
\begin{equation}
\label{vol-area}
\mathrm{area}(S) = \frac{1}{2\pi} \mathrm{vol}(T^1 S) = \frac{1}{2\pi} \int_{\R/2\pi \Z\times \R/2\pi \Z \times (0,M/2)} \alpha\wedge d\alpha = 2\pi \int_0^{M/2} r(s)\, ds.
\end{equation}
The Reeb vector field of the contact form $\alpha$ has the expression
\[
R(\theta,s,\beta) = \frac{\cos \beta}{r(s)} \frac{\partial}{\partial \theta}  + \frac{r'(s) \cos \beta}{r(s)} \frac{\partial}{\partial \beta} + \sin \beta \frac{\partial}{\partial s}.
\]
Therefore, the geodesic equation takes the form of the following system 
\begin{eqnarray}
\label{geoeq1}
\dot\theta &=& \frac{\cos \beta}{r(s)}  \\ 
\label{geoeq3}
\dot\beta &=& \frac{r'(s) \cos \beta}{r(s)}
\\
\label{geoeq2}
\dot{s} &=& \sin \beta.
\end{eqnarray}
From the last two equations one immediately recovers the familiar fact that the {\em Clairaut function}
\[
K: T^1 S \rightarrow \R, \qquad K(u) = \left\{ \begin{array}{ll} K(\theta,\beta,s) := r(s) \cos \beta & \mbox{if } u\notin T_{p_s} S \cup T_{p_N} S, \\ 0 & \mbox{if } u\in T_{p_s} S \cup T_{p_N} S, \end{array} \right.
\]
is a first integral of the geodesic flow. The invariance of the Clairaut function implies the following well known facts about the asymptotic behaviour of geodesics other than meridians:

\begin{lem}
\label{gengeo}
Let $\gamma: \R \rightarrow S$ be a geodesic which is not a meridian. Then exactly one of the following two alternative conditions hold:
\begin{enumerate}[(i)]
\item for $t\rightarrow -\infty$ and $t\rightarrow +\infty$ the geodesic $\gamma$ is asymptotic to two possibly coinciding equators $P_{s_-}$ and $P_{s_+}$ with $r(s_-) = r(s_+)= |K(\dot\gamma)|$;
\item there exist numbers $0<s_1<s_2<M/2$ such that 
\[
r(s_1)=r(s_2)=|K(\dot\gamma)| < r(s) \quad \forall s\in (s_1,s_2), \quad r'(s_1)>0, \; r'(s_2)<0,
\]
$\gamma$ is confined to the strip
\[
\bigcup_{s\in [s_1,s_2]} P_s,
\]
and it alternately touches both parallels $P_{s_1}$ and $P_{s_2}$ tangentially infinitely many times.
\end{enumerate}
\end{lem}

\begin{proof}
Since $\gamma$ is not a meridian, the Clairaut integral $K(\dot{\gamma})$ does not vanish; without loss of generality we may assume that $K(\dot{\gamma})>0$. 

The fact that $\gamma$ does not run through the poles - meridians are the only geodesics doing this - allows us to express its derivative in terms of the coordinate system $\Psi$ introduced above as
\[
\dot\gamma(t) = \bigl(\theta(t),\beta(t),s(t)\bigr),  \qquad \forall t\in \R,
\]
for suitable smooth functions $\theta: \R \rightarrow \R/2\pi \Z$,  $\beta: \R\rightarrow \R/2\pi \Z$ and $s: \R \rightarrow (0,M/2)$ satisfying the equations (\ref{geoeq1}), (\ref{geoeq3}) and (\ref{geoeq2}).

\medskip

\noindent{\em Claim 1. If there is a $t_0\in \R$ such that $\dot{s}(t)\neq 0$ for all $t>t_0$, then for $t\rightarrow +\infty$ the geodesic $\gamma$ is asymptotic to an equator $P_{s_+}$ with $r(s_+)=K(\dot\gamma)$. An analogous result holds when $\dot{s}(t)\neq 0$ for all $t<t_0$.}

\medskip

Indeed, assume without loss of generality that $\dot{s}(t)>0$ for all $t>t_0$. Then $s$ is strictly increasing on the interval $[t_0,+\infty)$ and must converge to some $s_+\in (0,M/2]$ for $t\rightarrow +\infty$. The conservation of Clairaut's integral
\[
r(s(t)) \cos \beta(t) = K(\dot{\gamma}) >0
\]
ensures that $r\circ s$ is bounded away from 0 and hence $s_+<M/2$. The above identity also implies that $\beta(t)$ converges to some $\beta_+$ for $t\rightarrow +\infty$ with
\[
r(s_+) \cos \beta_+ = K(\dot{\gamma}).
\]
By equation (\ref{geoeq2}), $\dot{s}(t)$ converges to $\sin \beta_+$ for $t\rightarrow +\infty$, so the fact that $s(t)$ is increasing and converges for $t\rightarrow +\infty$ forces $\beta_+$ to be zero in $\R/2\pi \Z$. From (\ref{geoeq3}) we deduce that $\dot\beta(t)$ converges to $r'(s_+)/r(s_+)$ for $t\rightarrow +\infty$, and the fact that $\beta(t)$ converges for $t\rightarrow +\infty$ implies that $r'(s_+)=0$. We conclude that
\[
\dot{\theta}(t) \rightarrow \frac{1}{r(s_+)}, \quad \beta(t)\rightarrow 0 \mod 2\pi , \quad s(t) \rightarrow s_+ \qquad \mbox{for } t\rightarrow +\infty,
\]
and hence $\gamma$ is asymptotic to the equator $P_{s_+}$ for $t\rightarrow +\infty$. This concludes the proof of Claim 1.

\medskip

In what follows, we assume that $\gamma$ is not an equator, since this case is covered by condition (i). 

\medskip

\noindent{\em Claim 2. Assume that $t_0\in \R$ is such that $\dot{s}(t_0)=0$. Then $r'(s(t_0))\neq 0$, $t_0$ is an isolated zero of $\dot{s}$, and $s$ has a strict local maximum or a strict local minimum at $t_0$. }

\medskip

Assume that $t_0\in \R$ is such that $\dot{s}(t_0)=0$. By equation (\ref{geoeq2}), we have $\beta(t_0)\in \pi \Z$. If $r'(s(t_0))=0$, then uniqueness implies that $\gamma$ parametrizes the equator $P_{s(t_0)}$, and we are excluding this case. Therefore, $r'(s(t_0))\neq 0$. Then (\ref{geoeq3}) implies that $\beta$ is strictly monotone in a neighborhood of $t_0$. Therefore, (\ref{geoeq2}) implies that $\dot{s}(t)$ changes sign when $t$ crosses $t_0$, and hence $t_0$ is an isolated zero of $\dot{s}$ and $s$ has a strict local maximum or a strict local minimum at $t_0$. This concludes the proof of Claim 2.

\medskip

We now argue on the number of zeros of $\dot{s}$. If $\dot{s}$ has no zeros or just one zero, then by Claim 1 we get that alternative (i) holds. There remains to consider the case in which $\dot{s}$ has at least two zeros. Let $t_1<t_2$ be two consecutive zeros of $\dot{s}$ and assume without loss of generality that $\dot{s}>0$ on $(t_1,t_2)$. Then $s$ increases strictly monotonically from $s_1:=s(t_1)$ to $s_2:=s(t_2)$ on $[t_1,t_2]$. The conservation of Clairaut's integral implies that
\[
K(\dot\gamma) = r(s_1) = r(s_2) = r(s(t)) \cos \beta(t) \qquad \forall t\in \R.
\]
The fact that $\dot{s}>0$ on $(t_1,t_2)$ and equation (\ref{geoeq2}) imply that $\cos\beta<1$ on $(t_1,t_2)$ and hence the above identity shows that
\begin{equation}
\label{campana}
r(s)>r(s_1)=r(s_2) \qquad \forall s\in (s_1,s_2).
\end{equation}
By Claim 2, $s_1$ and $s_2$ are not critical points of $r$ and since they minimize $r$ on the interval $[s_1,s_2]$ we must have
\begin{equation}
\label{diversi}
r'(s_1)>0 \qquad \mbox{and} \qquad r'(s_2)<0.
\end{equation}
Again by Claim 2, $s$ achieves a strict local maximum at $t_2$ and starts decreasing again after $t_2$. As long as $s(t)$ stays above $s_1$, $r(s(t))$ remains strictly larger that $r(s_1)$ by (\ref{campana}) and hence the conservation of Clairaut's integral and (\ref{geoeq2}) imply that $\dot{s}(t)$ remains negative. The function $s$ cannot remain strictly above $s_1$ for all $t>t_2$, because in this case by Claim~1 it would converge to a critical value $s\in [s_1,s_2]$ of $r$ with $r(s)=r(s_1)$ (again by the conservation of Clairaut's integral), and there are no such points in $[s_1,s_2]$, by (\ref{campana}) and (\ref{diversi}). We conclude that $s$ must reach the value $s_1$ at some first instant $t_3>t_2$. By Claim~2, $s$ achieves a strict local minimum at $t_3$. By iterating this argument and by an analogous analysis for $t<t_1$, we see that $s$ is forced to oscillate infinitely many times between the values $s_1$ and $s_2$, and that the set of zeros of $\dot{s}$ is unbounded from above and from below and consists of global minimizers  and maximizers  of $s$, with values $s_1$ and $s_2$ respectively. At these instants, $\beta$ belongs to $\pi \Z$ and hence the geodesic is tangent to the parallels $P_{s_1}$ and $P_{s_2}$. This shows that alternative (ii) holds.
\end{proof}

\section{Birkhoff annuli at equators of minimal length}
\label{biran}

Let $s_0\in (0,M/2)$ be a critical point of $r$. 
Let $A$ be the {\em open Birkhoff annulus} associated to the corresponding positively oriented equator: $A$ is the set of unit tangent vectors $u$ based at points on $P_{s_0}$ such that $\beta(u)\in (0,\pi)$. The geodesic flow is transverse to the open annulus $A$.

\begin{lem}
\label{welldef}
Assume that $s_0\in (0,M/2)$ is a critical point of $r$ such that $r(s_0)$ is the minimum over all critical values of $r$ on $(0,M/2)$. Then the forward and backward evolutions of any vector in the corresponding Birkhoff annulus $A$ meet $A$ again.
\end{lem}

\begin{proof}
Let $u$ be an element of $A$ and denote by $\gamma_u:\R \rightarrow S$ the geodesic such that $\dot\gamma_u(0)=u$. Then
\begin{equation}
\label{sotto}
|K(u)|=r(s_0) |\cos \beta(u)| < r(s_0).
\end{equation}
We must show that $\gamma_u$ hits the equator $P_{s_0}$ with positive derivative of its $s$ component both in the future and in the past. This is certainly true if this geodesic is a meridian, so there remains to consider the two cases (i) and (ii) from Lemma \ref{gengeo}. 

Case (i) cannot occur: In this case, $\gamma_u$ would be asymptotic to an equator $P_{s_+}$, $s_+\in (0,M/2)$, with $r(s_+)=|K(u)|$ and by (\ref{sotto}) the point $s_+\in (0,M/2)$ would be a critical point of $r$ with $r(s_+)<r(s_0)$, contradicting our choice of $s_0$. 

Therefore, case (ii) occurs and $\gamma_u$ is confined in a strip
\[
\bigcup_{s\in [s_1,s_2]} P_s,
\]
and it alternately touches each parallel $P_{s_1}$ and $P_{s_2}$, which are not equators, infinitely many times in the past and in the future. The fact that $\gamma_u(0)$ belongs to $P_{s_0}$ forces $s_0$ to belong to the open interval $(s_1,s_2)$. Then the fact that $\gamma_u$ touches $P_{s_1}$ and $P_{s_2}$ infinitely many times in the past 
and in the future and that it can never be tangent to $P_{s_0}$ implies that $\gamma_u$ hits the equator $P_{s_0}$ with positive derivative of its $s$ component infinitely many times in the past and in the future. This concludes the proof.
\end{proof} 

Let $s_0\in (0,M/2)$ be a critical point of $r$ with minimal value of $r$ as in the lemma above and let $A$ be the Birkhoff annulus corresponding to the positively oriented equator $P_{s_0}$. By the above lemma, the first return map and first return time to $A$ are well defined:
\[
\begin{split}
\tau: A \rightarrow (0,+\infty), &\qquad \tau(u) := \min \{t\in (0,+\infty) \mid \phi^t(u)\in A\}, \\
\varphi: A \rightarrow A, &\qquad \varphi(u) = \phi^{\tau(u)}(u).
\end{split}
\]
By the transversality of the geodesic flow to $A$, the function $\tau$ is smooth on $A$ and $\varphi: A \rightarrow A$ is a smooth diffeomorphism. 

Notice that the value of the Clairaut integral at every $u\in A$ satisfies the inequality
\[
|K(u)| = r(s_0) |\cos \beta(t)| < r(s_0).
\]
Therefore, the flow saturation of the annulus $A$, that is the set
\[
\bigcup_{t\in \R} \phi^t(A),
\]
is disjoint from the set of unit vectors $u\in T^1 S$ with $|K(u)|\geq r(s_0)$. When $r(s_0)$ is the maximum of $r$ - this happens if and only if the function $r$ has a unique positive critical value - the latter set is precisely the set of unit vectors which are tangent to the equators $P_s$ with $r(s)=\max r$; these values of $s$ form a closed interval containing $s_0$ which is just the singleton $\{s_0\}$ when $s_0$ is a strict maximizer of $r$. When the maximum of $r$ is larger than $r(s_0)$, the complement of the flow saturation of $A$ has non-empty interior and hence positive contact volume. Actually, we can show that the complement of the  flow saturation of $A$ consists precisely of those vectors $u$ in $T^1 S$ for which $|K(u)|\geq r(s_0)$. Equivalently, the following result holds:

\begin{lem}
\label{complemento}
Assume that $s_0\in (0,M/2)$ is a critical point of $r$ such that $r(s_0)$ is the minimum over all critical values of $r$ on $(0,M/2)$ and let $A$ be the open Birkhoff annulus corresponding to the positively oriented equator $P_{s_0}$. Then
\[
\bigcup_{t\in \R} \phi^t(A) = \{ u\in T^1 S \mid |K(u)|< r(s_0)\}.
\]
\end{lem}

\begin{proof}
The inclusion $\subset$ having  been already clarified, we must show that the orbit of any $u\in T^1 S$ with $|K(u)|<r(s_0)$ by the geodesic flow hits  the Birkhoff annulus $A$. Let $\gamma: \R \rightarrow S$ be the geodesic determined by $u$. If $K(u)=0$ then $\gamma$ is parametrizing a meridian and the orbit of $u$ certainly meets $A$. Therefore, we may assume that $\gamma$ is not a meridian, so one of the two alternative conditions of Lemma \ref{gengeo} holds. If (i) holds, then there is an equator $P_{s_+}$ with $r(s_+)=|K(u)|<r(s_0)$, and this is not possible because we are assuming that $r(s_0)$ is the minimum over all positive critical values of $r$. Therefore, condition (ii) must hold and the geodesic $\gamma$ alternately touches the two parallels $P_{s_1}$ and $P_{s_2}$, where $0<s_1<s_2<M/2$ and
\[
r(s_1)=r(s_2) = |K(u)| < r(s_0).
\]
The above inequality implies that $s_1<s_0<s_2$: if not, assuming without loss of generality that $s_1<s_2<s_0$, we would get that the minimum of $r$ on $[s_1,s_0]$ must be achieved at some point $s\in (s_1,s_0)$, and this point would be a critical point of $r$ with $r(s)<r(s_0)$. The fact that $s_0$ is between $s_1$ and $s_2$ ensures that $\gamma$ crosses the equator $P_{s_0}$ in the direction of increasing values of $s$, and at this instant the orbit of $u$ meets $A$.
\end{proof}

\section{The generating function}

Let $s_0\in (0,M/2)$ be a critical point of $r$ such that $r(s_0)$ is the minimum over all critical values of $r$ on $(0,M/2)$ and let $A$ be the open Birkhoff annulus corresponding to the positively oriented equator $P_{s_0}$, as in the previous section. The elements of $A$ can be parametrized by pairs $(\xi,\beta)\in \R/L \Z \times (0,\pi)$. Here $L=2\pi r(s_0)$ denotes the length of the equator $P_{s_0}$, the variable
\[
\xi:= r(s_0) \theta \in \R/L \Z
\]
parametrizes the base point of $u\in A$ by arc length in the positive direction, $\theta\in \R/2\pi \Z$ being the angular coordinate as in the previous section, and $\beta$ is the angle between the unit tangent vector and the positive direction of the equator, again as in the previous section. It will be convenient to use coordinates $(\xi,\eta)\in \R/L \Z \times (-1,1)$ on $A$, where $\eta = - \cos \beta$.

By the identity (\ref{contact}), the Hilbert contact form of $T^1 S$ restricts to the following 1-form on $A$:
\[
\lambda := \cos \beta \, d\xi = - \eta\, d\xi.
\]
This 1-form is a primitive of the area form 
\begin{equation}
\label{omega}
\omega := \sin \beta \, d\xi \wedge d\beta = d\xi\wedge d\eta 
\end{equation}
on $A$ and satisfies
\begin{equation}
\label{fund}
\varphi^* \lambda - \lambda = d \tau.
\end{equation}
The latter identity relating the first return map and the first return time is a general feature of surfaces of section in three-dimensional Reeb dynamics. Its simple proof can be found for instance in     
\cite[Section 3.2]{abhs17} and \cite[Equation (98)]{abhs18}.
By differentiating this identity, we see that $\varphi$ preserves the area form $\omega$. 

Since the restriction of the Clairaut integral to $A$ is $r(s_0) \cos \beta = - r(s_0) \eta$, the second component of $(\xi,\eta)$ is preserved by the first return map $\varphi$. By the rotational symmetry of $S$, $\varphi$ commutes with the translations on $\R/L \Z$. We conclude that $\varphi$ has the following form
\[
\varphi(\xi,\eta) = (\xi+f(\eta),\eta) \qquad \forall (\xi,\eta) \in \R/L \Z \times (-1,1),
\]
where $f: (-1,1) \rightarrow \R$ is a smooth function which is uniquely defined up to the sum of a multiple of $L$. The fact that meridians are closed geodesics implies that the points of the form $(\xi,0)$ are fixed by $\varphi$, and hence we can normalize $f$ by requiring:
\[
f(0)=0.
\]
The symmetry of $S$ with respect to reflections with respect to planes containing the $z$-axis implies that $f$ is an odd function.  

Let $F: (-1,1) \rightarrow \R$ be a primitive of $f$. Notice that
\[
\varphi^* \lambda - \lambda = -\eta \, d(\xi+f(\eta)) + \eta\, d\xi = - \eta f'(\eta)\, d\eta = - \eta F''(\eta) \, d\eta = d \bigl( F(\eta) - \eta F'(\eta) \bigr).
\]
The above identity and (\ref{fund}) imply that the functions $\tau$ and $(\xi,\eta) \mapsto F(\eta) - \eta F'(\eta)$ differ by a constant. We can therefore normalize the primitive $F$ of $f$ in such a way that
\[
\tau(\xi,\eta) = F(\eta) - \eta F'(\eta) \qquad \forall (\xi,\eta)\in \R/L \Z \times (-1,1).
\]
In other words, we are normalizing $F$ in such a way that $F(0)$ is the length $M$ of the meridians.
Being a primitive of an odd function, the function $F$ is even. We summarize the above discussion in the following:

\begin{lem}
\label{genfun}
Let $A \cong \R/L \Z \times (-1,1)$ be the Birkhoff annulus associated to a positively oriented equator having minimal length $L$ among all equators. Then the first return map $\varphi: A \rightarrow A$ and first return time $\tau: A \rightarrow (0,+\infty)$ to $A$ have the form
\[
\varphi(\xi,\eta) = (\xi+F'(\eta),\eta), \qquad \tau(\xi,\eta) = F(\eta) - \eta F'(\eta),
\]
where $F:(-1,1) \rightarrow \R$ is an even smooth function.
\end{lem}

We shall refer to the above function $F$ as the {\em generating function} of the first return map $\varphi$.

\section{Properties of the generating function}

We fix an equator $P_{s_0}$ which is assumed to have minimal length $L = 2 \pi r(s_0)$ among all equators. We denote by $A$ the corresponding open Birkhoff annulus and we endow it with coordinates $(\xi,\eta)\in \R/L \Z \times (-1,1)$ as in the previous section. The first return time to $A$ is denoted by $\tau:A \rightarrow \R$, the first return map by $\varphi: A \rightarrow A$, and its generating function by $F:(-1,1)\rightarrow \R$. In this section, we want to establish some further properties of $F$ and to express some geometric and dynamical quantities in terms of $F$.

Let $u\in A$ be given by $(\xi,\eta)\in \R/L \Z \times (-1,1)$ with $\eta\neq 0$. Then the geodesic $\gamma_u$ with initial vector $u$ is not a meridian and hence does not touch the $z$-axis. As such, its portion $\gamma_u|_{[0,\tau(u)]}$ has a winding number $W(u)$ with respect to the $z$-axis: $W(u)$ is the real number
\[
W(u) := \frac{\theta(\tau(u))-\theta(0)}{2\pi},
\]
where $\theta: \R \rightarrow \R$ is a continuous function such that
\[
\gamma_u(t) = \bigl(r(s(t)) \cos \theta(t), r(s(t))\sin \theta(t),z(s(t))\bigr) \qquad \forall t\in \R,
\]
for a suitable smooth function $s: \R \rightarrow (0,M/2)$. The rotational symmetry of $S$ implies that the winding number $W(u)$ depends only on $\eta\in (-1,1)\setminus \{0\}$, and hence we shall also indicate it by $W(\eta)$.

\begin{lem}
\label{winding}
The winding number $W$ and the generating function $F$ are related by the identities
\[
\begin{split}
F'(\eta) &= L W(\eta) + L \qquad \forall \eta\in (0,1), \\
F'(\eta) &= L W(\eta) - L \qquad \forall \eta\in (-1,0).
\end{split}
\]
\end{lem}

\begin{proof}
From the form $\xi+f(\eta)=\xi + F'(\eta)$ of the first component of $\varphi(\xi,\eta)$ and from the definition of $W(\eta)$ it follows that the quantity
\[
F'(\eta) - L W(\eta) = f(\eta) - L W(\eta)
\]
is an integer multiple of $L$. By continuity of both $f$ and $W$ on the interval $(-1,0)$, we deduce that there is some integer $m$ such that
\begin{equation}
\label{intero}
F'(\eta) - L W(\eta) = m L \qquad \forall \eta\in (-1,0).
\end{equation}
When $\eta$ is negative and close to zero, the corresponding angle $\beta=\arccos (-\eta)$  is close to $\pi/2$ and smaller than $\pi/2$. By equation (\ref{geoeq1}), the time derivative of the function $\theta(t)$ is positive, so the winding number $W(\eta)$ must be positive and bounded away from zero as $\eta$ tends to zero from below. As $\eta$ converges to $0$,  $\tau(u)$ converges to the length $M$ of the meridian and the projection of the geodesic segment $\gamma_u|_{[0,\tau(u)]}$ onto the $x,y$-plane, that is the curve 
\[
t\mapsto (r(s(t)) \cos \theta(t), r(s(t)) \sin \theta(t) ),
\]
converges in the $C^{\infty}$-topology to the projection of the arc-length parametrization of the meridian on the interval $[0,M]$. The latter curve is closed, has as image a segment,  and crosses the origin exactly twice with non-zero speed. Any $C^1$-small perturbation of this curve avoiding the origin will have winding number close to either 1, -1 or 0. Being positive and bounded away from zero, $W(\eta)$ is close to $1$ when $\eta$ is close to zero and negative. Together with the fact that $F'=f$ is continuous and vanishes at $0$, we obtain that the integer $m$ appearing in (\ref{intero}) has the value $-1$, and hence
\[
F'(\eta) = L W(\eta) - L \qquad \forall \eta\in (-1,0).
\]
The analogous formula for $\eta\in (0,1)$ follows from the fact that both $F'$ and $W$ are odd functions.
\end{proof}

The above Lemma can be used to bound the function $F$ from below:

\begin{lem}
\label{stima}
Let $A\cong \R/L \Z \times (-1,1)$ be the Birkhoff annulus associated to a positively oriented equator $P_{s_0}$ having minimal length $L$ among all equators. Then the function $F$ from Lemma \ref{genfun} has the lower bound
\[
F(\eta) > L |\eta| \qquad \forall \eta\in (-1,1).
\]
\end{lem}
   
\begin{proof}
Fix some $\xi\in \R/L \Z$ and $\eta\in (0,1)$. The corresponding unit vector $u\in A$ makes an angle $\beta= \arccos (-\eta) \in (\pi/2,\pi)$ with the positive direction of the equator. The value of the Clairaut integral on the geodesic emanating from $u$ is the negative number
\[
r(s_0) \cos \beta = - r(s_0) \eta.
\]
By the invariance of the Clairaut integral, the distance of this geodesic from the $z$-axis never gets smaller than the absolute value of the above number, that is $r(s_0) \eta$. The geodesic arc from $u$ to $\varphi(u)$ has length
\begin{equation}
\label{tau}
\tau(u) = F(\eta) - \eta F'(\eta).
\end{equation}
Since this geodesic arc is not contained in a horizontal plane, its length is strictly larger then the length of its projection onto the $(x,y)$-plane. Since this projected curve remains at distance at least $r(s_0)\eta$ from the origin and has winding number $W(\eta)$ around this point, its length is not smaller than
\[
2 \pi r(s_0) \eta |W(\eta)| = L \eta |W(\eta)|.
\]
Therefore, we have
\[
\tau(u) >  L\eta |W(\eta)| \geq - L \eta W(\eta).
\]
By the first formula in Lemma \ref{winding} we find
\[
\tau(u) > - \eta( f(\eta) - L ) = - \eta F'(\eta) + L \eta.
\]
Together with (\ref{tau}) we conclude that
\[
F(\eta) > L \eta \qquad \forall \eta\in (0,1).
\]
The desired conclusion follows from the fact that $F$ is even and from the bound $F(0)>0$, which holds because $F(0)$ is the length of meridians. 
\end{proof}

In order to understand the behaviour of the generating function $F$ near $-1$ and $1$ we shall express it in terms of an integral formula involving the area of suitable superlevels of the Clairaut integral. We start by noticing that, being independent of $\theta$, the Clairaut integral defines a function on the closed annulus $\R/2\pi \Z\times [0,M/2]$, which we still denote by $K$:
\[
K(\beta,s) = r(s) \cos \beta.
\]
By symmetry, we can restrict attention to the following closed rectangle $\overline{Q}$, where
\[
Q := \left(-\frac{\pi}{2},\frac{\pi}{2}\right) \times \left(0,\frac{M}{2}\right).
\]
Notice that $K$ is positive on the open rectangle $Q$ and vanishes on its boundary. The formula
\[
dK(\beta,s) = - r(s) \sin \beta \, d\beta +  r'(s) \cos \beta \, ds
\]
shows that the critical points of $K$ in $Q$ are precisely the points $(0,s)$, where $s$ is a critical point of $r$. The critical values of $K|_Q$ are exactly the critical values of $r$. Since $s_0$ is a critical point of $r|_{(0,M/2)}$ with minimal value among all critical points, the interval $(0,r(s_0))$ consists of regular values for $K|_Q$. For every $\kappa\in (0,r(s_0))$ the level set $K^{-1}(\kappa)$ has the form
\[
K^{-1}(\kappa) = \left\{ (\beta,s) \in Q \mid \beta = \pm \arccos \frac{\kappa}{r(s)}, \; s\in r^{-1}([\kappa,+\infty)) \right\},
\]
where $r^{-1}([\kappa,+\infty))\subset (0,M/2)$ is a compact interval because the positive number $\kappa$ is smaller than the smallest critical value of $r$. Therefore, for these values of $\kappa$ the level set $K^{-1}(\kappa)$ is an embedded circle bounding  the set
\[
\Omega_{\kappa} := \{ (\beta,s)\in Q \mid K(\beta,s)>\kappa \},
\]
which is diffeomorphic to an open disk. The set $\Omega_0$ is the whole $Q$, and the sets $\Omega_{\kappa}$ for $\kappa<r(s_0)$ form a fundamental system of open neighborhoods of the compact set
\begin{equation}
\label{Gamma}
\Gamma := \{ (\beta,s)\in Q \mid K(\beta,s)\geq r(s_0) \}.
\end{equation}

\begin{lem}
\label{lemma-area}
The generating function $F$ can be expressed by the identity
\begin{equation}
\label{integral}
F(\eta) = \int_{\Omega_{\kappa(\eta)}} \cos \beta \, d\beta \wedge ds + L |\eta| \qquad \forall \eta\in (-1,1),
\end{equation}
where $\kappa(\eta):=|K(u)| = r(s_0) |\cos \beta(u)| = r(s_0) |\eta|$ is the absolute value of the Clairaut function $K$ on the unit tangent vector $u=(\xi,\eta)\in A$. In particular, $F$ extends continuously to the closed interval $[-1,1]$ by setting
\[
F(-1)=F(1)=\int_{\Gamma} \cos \beta \, d\beta \wedge ds + L.
\]
\end{lem}

\begin{proof}
Since $F$ is a continuous even function on $(-1,1)$ and so is the expression on the right hand side of (\ref{integral}), it is enough to prove this identity for $\eta\in (-1,0)$, that is, when the angle $\beta(u)$ which the unit tangent vector $u\in A$ makes with the positive direction of the equator $P_{s_0}$ belongs to $(0,\pi/2)$ (recall that the cosine of this angle is $-\eta$). In this case, the Clairaut integral $K(u)=r(s_0) \cos \beta(u)$ is positive and hence
\[
\kappa = |K(u)| = K(u) = r(s_0) \cos \beta(u).
\]
Therefore, the orbit of $u$ is contained in the region of $T^1 S$ which is given by the image of 
\[
\R/2\pi \Z \times \left(-\frac{\pi}{2},\frac{\pi}{2} \right) \times \left(0,\frac{M}{2} \right) = \R/2\pi \Z \times Q
\]
by the parametrization $\Psi$ which is introduced in (\ref{Psi}). The function $(\theta,\beta,s)\mapsto r(s)/\cos \beta$ is positive on that region and if we multiply the Reeb vector field $R$ by this function we obtain
\begin{equation}
\label{newR}
\widehat{R}(\theta,\beta,s) := \frac{r(s)}{\cos \beta} R (\theta,\beta,s) = \frac{\partial}{\partial \theta} + r'(s) \frac{\partial}{\partial \beta} + r(s) \tan \beta \frac{\partial}{\partial s}.
\end{equation}
Denote by 
\[
(\widehat{\theta},\widehat{\beta},\widehat{s}): \R \rightarrow \R/2\pi \Z \times Q
\]
the orbit of $u=(\theta(u),\beta(u),s_0)\in A$ by the flow of $\widehat{R}$. This orbit is a time reparametrization of the orbit of $u$ by the geodesic flow. The latter orbit hits $A$ again for the first time at time $\tau(u)$; let $T>0$ be the instant corresponding to $\tau(u)$ in the new time reparametrization. From the fact that the time derivative of $\widehat{\theta}$ is constantly equal to 1, we deduce that $T$ coincides with the total variation of the variable $\theta$ along the portion of the orbit of the geodesic flow corresponding to $[0,\tau(u)]$, and by the definition of the winding number $W$ we obtain
\begin{equation}
\label{TT}
T = \widehat{\theta}(T) - \widehat{\theta}(0)=2\pi W(u) = 2 \pi W(\eta).
\end{equation}
The first return time $\tau(u)$ coincides with the integral of the Hilbert contact form on the curve $\phi^t(u)$, $t\in [0,\tau(u)]$. Since the integral of a one-form is independent of the parametrization, we obtain, using the formula (\ref{contact}), 
\begin{equation}
\label{tauint}
\tau(u) = \int_{(\widehat{\theta},\widehat{\beta},\widehat{s})|_{[0,T]}} \alpha = \int_{(\widehat{\theta},\widehat{\beta},\widehat{s})|_{[0,T]}} r(s) \cos \beta\, d\theta +  \int_{(\widehat{\theta},\widehat{\beta},\widehat{s})|_{[0,T]}} \sin \beta \, ds.
\end{equation}
We analyse these two integrals separately. Using the conservation of Clairaut's function and the identity (\ref{TT}), we deduce that
\[
\int_{(\widehat{\theta},\widehat{\beta},\widehat{s})|_{[0,T]}} r(s) \cos \beta\, d\theta = K(u) \int_{\widehat{\theta}|_{[0,T]}} d\theta = 2\pi K(u) W(u).
\]
The integrand in the last integral of (\ref{tauint}) is independent of $\theta$, so this is an integral over the curve $(\widehat{\beta},\widehat{s}):[0,T] \rightarrow \R/2\pi \Z \times (0,M/2)$. By the conservation of Clairaut's function, this curve takes values in $Q$, and more precisely in the level set
\[
K^{-1}(\kappa) = \{(\beta,s)\in Q \mid K(\beta,s)=\kappa\},
\]
where $\kappa=K(u)>0$. The projection of the annulus $A$ to the rectangle $Q$ is given by the open segment consisting of points of the form $(\beta,s)\in Q$ with $\beta\in (0,\pi/2)$ and $s=s_0$. This segment intersects the level set $K^{-1}(\kappa)$, which as we have seen is an embedded circle, in exactly one point, namely $(\beta(u),s_0)$. Since $\tau(u)$ is the first return time to $A$ and since the planar vector field of which $(\widehat{\beta},\widehat{s})$ is an integral curve does not vanish on $K^{-1}(\kappa)$, the curve $(\widehat{\beta},\widehat{s}):[0,T] \rightarrow \R/2\pi \Z \times (0,M/2)$ is closed and is a simple parametrization of the embedded circle $K^{-1}(\kappa)$. The fact that the time derivative of $\widehat{s}$ at 0 is positive implies that this parametrization preserves the counterclockwise orientation of $K^{-1}(\kappa)$ in the $(\beta,s)$-plane. As we have seen, the embedded circle $K^{-1}(\kappa)$ is the boundary of $\Omega_{\kappa}$, so by Stokes theorem we obtain the identity
\[
\int_{(\widehat{\theta},\widehat{\beta},\widehat{s})|_{[0,T]}} \sin \beta \, ds = \int_{(\widehat{\beta},\widehat{s})|_{[0,T]}} \sin \beta \, ds = \int_{\Omega_{\kappa}} \cos \beta \, d\beta\wedge ds.
\]
Therefore, (\ref{tauint}) can be rewritten as
\[
\tau(u) = 2\pi K(u) W(u) + \int_{\Omega_{\kappa}} \cos \beta \, d\beta\wedge ds.
\]
By expressing $\tau$ and $W$ in terms of the generating function $F$ and its derivative $F'=f$ as in Lemmas \ref{genfun} and \ref{winding} and by the identity
\[
K(u) = r(s_0) \cos \beta(u) = - r(s_0) \eta = - \frac{L}{2\pi} \eta,
\]
we find 
\[
F(\eta) - \eta f(\eta) = - L \eta W(\eta) + \int_{\Omega_{\kappa}} \cos \beta \, d\beta\wedge ds = -\eta f(\eta) - L\eta + \int_{\Omega_{\kappa}} \cos \beta \, d\beta\wedge ds,
\]
from which we conclude that
\[
F(\eta) = L |\eta| + \int_{\Omega_{\kappa}} \cos \beta \, d\beta\wedge ds,
\]
for all $\eta\in (-1,0)$, as we wished to prove.
\end{proof}

\begin{rem}
Notice that when $r(s_0)$ is the unique critical value of $r|_{(0,M/2)}$, necessarily its global maximum, the set $\Gamma$ consists of the pairs $(\beta,s)$ with $\beta=0$ and $r(s)=\max r$, and hence $\Gamma$ has empty interior and the integral of the area form $\cos \beta \, d\beta\wedge ds$ on $\Gamma$ vanishes. So in this case the function $F$ takes the value $L$ at $-1$ and $1$. Instead, when $r$ has more critical values then $r(s_0)<\max r=\max K$ and the set $\Gamma$ has a non-empty interior, then the integral of the area form $\cos \beta \, d\beta\wedge ds$ on $\Gamma$ is positive and $F(-1)=F(1)$ is strictly larger than $L$.
\end{rem}

\begin{rem}
We remark that the time reparametrization of the restriction of the Reeb flow to the region where $K>0$ which we considered in the above proof - see (\ref{newR}) - has the following nice properties: The time derivative of the angular component $\theta$ is 1, so $\theta$ is $2\pi$-periodic, while the components $(\beta,s)$ form integral lines of the autonomous planar Hamiltonian vector field $X_K$ on $Q$ which is induced by the Hamiltonian $K:Q \rightarrow \R$ and the symplectic form $\omega_Q:=\cos \beta \, d\beta\wedge ds$, meaning that
\[
\omega_Q (X_K,\cdot) = dK.
\]
Analogous facts hold on the region where $K<0$, when we multiply the Reeb vector field by the positive function $r(s)/|\cos \beta|$. 
\end{rem}

We can now express the contact volume of $T^1 S$ in terms of the generating function $F$ and a suitable integral over the set $\Gamma$:

\begin{lem}
\label{totvol}
Let $F:[-1,1]\rightarrow \R$ be the generating function of the first return map to the Birkhoff annulus of an equator of minimal length $L$, and let $\Gamma$ be the set which is defined in (\ref{Gamma}). Then the contact volume of $T^1 S$ takes the value
\[
\mathrm{vol}(T^1S) = 4 L \int_0^1 F(\eta)\, d\eta - 2 L^2 + \int_{\Gamma} (4\pi r(s) - 2 L \cos \beta)\, d\beta\wedge ds.
\]
\end{lem}

\begin{proof}
Denote by 
\[
\tilde{A}:=  \bigcup_{t\in \R} \phi^t(A)
\]
the open invariant subset of $T^1 S$ which is generated by the Birkhoff annulus $A$. By a standard argument, the contact volume $\tilde{A}$ equals the integral of $\tau$ on $A$ with respect to the area form $\omega$, see for instance \cite[Lemma 3.5]{abhs18}. By using the expressions for $\tau$ and $\omega$ in the coordinates $(\xi,\eta)\in \R/L\Z \times (-1,1)$ - see (\ref{omega}) and Lemma \ref{genfun} - we obtain from integration by parts and Lemma \ref{lemma-area}
\[
\begin{split}
\mathrm{vol}(\tilde{A}) &= \int_{A} \tau \, \omega = L \int_{-1}^1 \bigl( F(\eta) - \eta F'(\eta) \bigr) \, d\eta \\ &= 2 L \int_{-1}^1 F(\eta) \, d\eta - L \bigl( F(1) + F(-1) \bigr) \\ &= 4 L \int_0^1 F(\eta)\, d\eta - 2 L^2 -2 L \int_{\Gamma} \cos \beta \, d\beta\wedge ds.
\end{split}
\]
By Lemma \ref{complemento}, the contact volume of the complement of $\tilde{A}$ is
\[
\begin{split}
\mathrm{vol} ( T^1 S \setminus \tilde{A} ) &= \mathrm{vol} \bigl( \{ u\in T^1 S \mid |K(u)| \geq r(s_0)\} \bigr) \\ &= 2 \, \mathrm{vol} \bigl( \{ u\in T^1 S \mid K(u) \geq r(s_0)\} \bigr) \\ &= 2 \, \mathrm{vol} ( \R/2\pi \Z \times \Gamma) \\ &= 2 \int_{\R/2\pi \Z \times \Gamma} r(s)\, d\theta\wedge d\beta \wedge ds \\ &= 4\pi \int_{\Gamma} r(s)\, d\beta \wedge ds.
\end{split}
\]
Here we have used the expression (\ref{volume-form}) for the contact volume form. By adding these two identities we get the desired formula for the contact volume of $T^1 S$.
\end{proof}

We conclude this section by characterizing Zoll spheres of revolution in terms of the generating function $F$.

\begin{lem}
\label{carzoll}
The sphere of revolution $S$ is Zoll if and only if the generating function $F$ is constant. In this case, $F$ is constantly equal to $L$, $P_{s_0}$ is the unique 
equator of $S$, and the unique critical point $s_0$ of $r|_{(0,M/2)}$ is a non-degenerate global maximizer.
\end{lem}

\begin{proof}
First assume that $S$ is Zoll. Then all elements $(\xi,\eta)\in A$ are periodic points of $\varphi$. By the form of $\varphi$ given in Lemma \ref{genfun}, $(\xi,\eta)$ is a periodic point of $\varphi$ if and only if $F'(\eta) \in L \Z$. By smoothness, $F'$ is then forced to be a constant integer multiple of $L$. The fact that $F$ is even forces $F'$ to be zero and $F$ to be constant.

Now assume that $F$ is constant. Then $F(\eta)=M$ for every $\eta\in [-1,1]$, where $M=F(0)$ is the length of meridians, and from Lemma \ref{genfun} we deduce that $\varphi=\mathrm{id}$ and $\tau=M$ on $A$. For $b\in [0,\pi]$ we denote the orbit of the the geodesic flow starting at the unit vector $u_b$ with $\theta(u_b)=0$, $\beta(u_b)=b$ and $s(u_b)=s_0$ by
\[
\bigl( \theta_b(t), \beta_b(t), s_b(t) \bigr) = \phi^t(0,b,s_0).
\]
The fact that $\varphi$ is the identity and $\tau$ is constantly equal to $M$ on $A$ implies that these orbits are $M$-periodic for $b\in (0,\pi)$, and hence also for $b=0$ and $b=\pi$ (since the latter orbits are $L$-periodic, this forces $M$ to be a multiple of $L$). In particular the function 
\[
S(t):= \frac{\partial}{\partial b} s_b(t)\Big|_{b=0}
\]
is $M$-periodic. Together with the functions
\[
\Theta(t) :=  \frac{\partial}{\partial b} \theta_b(t)\Big|_{b=0} \qquad \mbox{and} \qquad  B(t) := \frac{\partial}{\partial b} \beta_b(t)\Big|_{b=0},
\]
the function $S$ gives us the solution $(\Theta,B,S)$ of the linearization of the system (\ref{geoeq1}-\ref{geoeq3}-\ref{geoeq2}) along the equator $P_0$, namely
\[
\dot{\Theta} (t) = 0, \qquad \dot{B}(t) = \frac{r''(s_0)}{r(s_0)} S(t), \qquad \dot{S}(t) = B(t),
\]
with initial conditions
\[
\Theta(0)=0, \qquad B(0)=1, \qquad S(0)=0.
\]
By the last two equations, the function $S$ satisfies the second order linear Cauchy problem
\begin{equation}
\label{eqS}
\ddot{S}(t) = \frac{r''(s_0)}{r(s_0)} S(t), \qquad S(0)=0, \qquad \dot{S}(0)=1.
\end{equation}
The solution of the above system is $M$-periodic only if $r''(s_0)$ is strictly negative. Therefore, $s_0$ is a non-degenerate local maximizer of the function $r$. Since $r(s_0)$ is the minimal critical value of $r|_{(0,M/2)}$, this implies that $s_0$ is the unique critical point of $r|_{(0,M/2)}$ and the global maximizer: $r(s)< r(s_0)$ for all $s\in (0,M)\setminus \{s_0\}$. We deduce that the set $\Gamma$ which is defined in (\ref{Gamma}) consists of the singleton $\{(0,s_0)\}$. Therefore, Lemma \ref{lemma-area} implies that $F(-1)=F(1)=L$, and hence $M=L$ and $F$ is constantly equal to $L$. Thus, all the the orbits of the geodesic flow which meet $A$ are closed and have period $L$, so they correspond to closed geodesics of length $L$.  By Lemma \ref{complemento}, all the orbits of the geodesic flow which do not meet $A$ belong to the set
\[
\{u\in T^1 S \mid |K(u)| \geq r(s_0)\} = \R/2\pi \Z \times \{(0,s_0),(\pi,s_0)\},
\]
where on the right we are using the standard coordinates $(\theta,\beta,s)$.
This set consists precisely in the orbits parametrizing the equators $P_{s_0}$ in both directions and with arbitrary starting point. Since also $P_{s_0}$ is a closed geodesic with length $L$, we deduce that $S$ is Zoll, as claimed. The other statements about the value of $F$ and the nature of the critical point $s_0$ have been proved along the way.
\end{proof}

\begin{rem}
Notice that equation (\ref{eqS}) is the equation for orthogonal Jacobi vector fields along $P_{s_0}$ and the number $r''(s_0)/r(s_0)$ appearing in it is minus the Gauss curvarure of $S$ along this equator. By studying the monotonicity regions of the functions $s_b$ it is easy to prove that $S$ has minimal period $L$ and $r''(s_0)=-1/r(s_0)$. So we recover the well known fact that each Zoll sphere of revolution has a unique equator and that along this equator the Gauss curvature coincides with that of the round sphere having the same equator.
\end{rem}

\section{Proof of Theorem \ref{mainthm1}}

We are now ready to prove Theorem \ref{mainthm1} from the introduction:

\setcounter{mainthm}{0}

\begin{mainthm}
The systolic ratio of a sphere of revolution $S$ does not exceed $\pi$. It equals $\pi$ if and only if $S$ is Zoll.
\end{mainthm}

\begin{proof}
Since Zoll spheres have systolic ratio $\pi$, we must prove that if a sphere of revolution $S$ is not Zoll than $\rho_{\mathrm{sys}}(S)<\pi$. Thanks to the first equality in the identity (\ref{vol-area}) we can work with the contact volume of $T^1S$ instead of the Riemannian area of $S$. Therefore, we have to prove the following statement: If a sphere of revolution $S$ is not Zoll, then it admits a closed geodesic $\gamma$ whose length $\ell(\gamma)$ satisfies
\begin{equation}
\label{tesi}
\ell(\gamma)^2 < \frac{1}{2} \mathrm{vol}(T^1 S).
\end{equation}

We represent $S$ as in Section \ref{sphererev} by the embedded curve $\sigma=(r,z):\R/M \Z\rightarrow \R^2$, where $M$ denotes the length of the meridians. Let $s_0\in (0,M/2)$ be a critical point of $r$ such that $r(s_0)$ is the minimum over all critical values of $r$ on $(0,M/2)$.   Let $A$ be the open Birkhoff annulus corresponding to the positively oriented equator $P_{s_0}$ and let $\varphi: A \rightarrow A$ and $\tau: A \rightarrow \R$ be the corresponding first return map and time function. The length of the equator $P_{s_0}$ is denoted by $L=2\pi r(s_0)$. We denote by $F:(-1,1) \rightarrow \R$ the generating function of $\varphi$ as in Lemma \ref{genfun} which, as we have seen in Lemma \ref{lemma-area}, extends continuously to $[-1,1]$.

By Lemma \ref{totvol} the contact volume of $T^1 S$ is given by the formula
\[
\mathrm{vol}(T^1S)= 4 L \int_0^1 F(\eta)\, d\eta - 2 L^2 + \int_{\Gamma} (4\pi r(s) - 2 L \cos \beta)\, d\beta\wedge ds,
\]
where $\Gamma\subset Q=(-\pi/2,\pi/2) \times (0,M/2)$ is the compact set
\[
\Gamma = \{ (\beta,s)\in Q \mid K(\beta,s) \geq r(s_0) \},
\]
$K(\beta,s)=r(s) \cos \beta$ denoting the Clairaut integral. On the set $\Gamma$, the function $r$ is certainly not smaller than $r(s_0)$, and hence the integrand of the latter integral is non-negative:
\[
4\pi r(s) - 2 L \cos \beta \geq 4 \pi r(s_0) - 2L \cos \beta = 2L  - 2L \cos \beta \geq 0 \qquad \mbox{on } \Gamma.
\]
Therefore, we have the inequality
\begin{equation}
\label{disvol}
\mathrm{vol}(T^1 S) \geq 4 L \int_0^1 F(\eta)\, d\eta - 2 L^2.
\end{equation}

Since the equator $P_{s_0}$ is a closed geodesic of length $L$, clearly $\ell_{\min}(S)\leq L$ and we may assume that
\[
L^2 \geq \frac{1}{2} \mathrm{vol} (T^1 S),
\]
because otherwise (\ref{tesi}) holds trivially. By the inequality (\ref{disvol}), this is equivalent to
\[
\int_0^1 F(\eta)\, d \eta \leq L.
\]
Since we are assuming that $S$ is not Zoll, the function $F$ is not constant, because of Lemma \ref{carzoll}.
Using also the fact that $F$ is a positive continuous function on $[-1,1]$ with $F(1)\geq L$ (see Lemma \ref{stima}), we deduce that the minimum $\mu$ of $F$ on $[0,1]$ is achieved in $[0,1)$ and belongs to the interval $(0,L)$. Since $F$ is an even function, $\mu$ is also the minimum of $F$ on $(-1,1)$, and hence it is a critical value of $F|_{(-1,1)}$. By the formulas for $\varphi$ and $\tau$ from Lemma \ref{genfun}, the surface $S$ has a closed geodesic $\gamma$ of length $\mu$. 

By Lemma \ref{stima}, we have
\[
F(\eta) \geq \max\{ \mu, L \eta\} \qquad \forall \eta\in [0,1].
\]
Since $F$ is differentiable at $\mu/L \in (0,1)$, the above inequality must be strict for $\eta=\mu/L$ and hence
\[
\int_0^1 F(\eta)\, d\eta > \int_0^1\max\{ \mu, L \eta\} \, d \eta = \mu + \frac{1}{2} \frac{(L-\mu)^2}{L}.
\]
By (\ref{disvol}) we have then
\[
\mathrm{vol}(T^1 S)\geq  4 L \int_0^1 F(\eta)\, d\eta - 2L^2 > 4 L \mu + 2(L-\mu)^2 - 2 L^2 = 2\mu^2.
\]
Therefore, $S$ has a closed geodesic $\gamma$ satysfying
\[
\ell(\gamma)^2 < \frac{1}{2} \mathrm{vol}(T^1 S),
\]
as we wished to prove.
\end{proof}

\section{Zermelo navigation data on surfaces of revolution}

In the following, we consider Finsler geodesic flows associated to a surface of revolution $S \subset \R^3$ and a rotational invariant killing vector field on $S$. We continue using the notation from  previous sections. In particular, $\Phi$ denotes the diffeomorphism onto $S\setminus \{p_S,p_N\}$ from Section \ref{sphererev}.

Given a real number $a$, let $W_a$ be the smooth vector field on $S$ given by 
\[
W_a  = a \left( x \frac{\partial}{\partial y} - y \frac{\partial}{\partial x} \right)\Big|_S,
\]  
which in the coordinates $(\theta,s)$ induced by $\Phi$ takes the form
\[
W_a  = a \, \frac{\partial}{\partial \theta}.
\]
We assume that $a$ satisfies
\begin{equation}\label{cond_Finsler}
|a| < \frac{1}{r_{\max}},
\end{equation}
where
\[
r_{\rm max}:=\max \{r(s)|s\in [0,M/2]\}
\]
denotes the maximal distance of a point of $S$ from the symmetry axis.  The triple 
\[
\left(S,\left< \cdot , \cdot \right>_{\R^3}|_S,W_a \right),
\]
which we denote simply by $S_a$, is called a Zermelo navigation data on $S$, and $W_a$ is referred to as the wind. The Zermelo navigation data gives rise to a Finsler metric $G_a$ on $S$ as follows. For each $p\in S$, consider the circle $T^1_pS_a\subset T_pS$ which is obtained by shifting the circle $T^1_pS$ by $W_a$:
\[
T_p^1 S_a := W_a(p) + T_p^1 S.
\]
Condition \eqref{cond_Finsler} implies that $0\in T_p^1S$ belongs to the bounded component of $T_p^1S \setminus T^1_pS_a$. In particular, the circles $\{T_p^1S_a \mid p\in S\}$ are the unit spheres of a Finsler metric 
\[
G_a: TS \rightarrow [0,+\infty).
\]
In other words, $G_a$ is the fiberwise positively homogeneous function taking the value 1 on each $T_p^1 S_a$.

\begin{rem}
It is possible to compute $G_a$ explicitly. If $\|\cdot\|$ denotes the norm induced by $\left<\cdot,\cdot\right>$ then in each tangent space $T_pS$ the function $u\mapsto \|u-G_a(u)W_a\|$ is $1$-homogeneous and realizes $T_p^1 S_a$ as its level one. By uniqueness $G_a(u) = \|u-G_a(u)W_a\|$ holds for every $u$. If we write $u=A\frac{\partial}{\partial \theta}+B\frac{\partial}{\partial s}$ then the above identity reads $\sqrt{r^2(A-G_a(u)a)^2+B^2}=G_a(u)$. Raising to the square and expanding we end up with a degree two equation for $G_a(u)$: 
\[
(1-r^2a^2)G_a^2(u) + 2r^2aAG_a(u) - (r^2A^2+B^2) = 0.
\]
Solving we get
\begin{equation}\label{explicit_G_a}
G_a\left(A\frac{\partial}{\partial \theta}+B\frac{\partial}{\partial s}\right) = \frac{\sqrt{r^2A^2+(1-r^2a^2)B^2}-r^2aA}{1-r^2a^2}.
\end{equation}
\end{rem}

The unit sphere bundle 
\[
T^1S_a = \bigcup_{p \in S} T_p^1S_a
\]
minus the two unit circles based at the north and south pole is now the image of the diffeomorphism
\[
\begin{split}
\Psi_a: \R/2\pi \Z  \times \R/2\pi \Z \times (0,M/2) \rightarrow  T^1 S \setminus ( T^1_{p_S} S \cup T^1_{p_N} S ), \\
\Psi_a(\theta,\beta,s) = \left( \Phi(\theta,s), \Bigl( a + \frac{1}{r} \cos \beta\Bigr) \frac{\partial \Phi}{\partial \theta} (\theta,s) + \sin \beta \frac{\partial \Phi}{\partial s} (\theta,s) \right),
\end{split}
\]
which we will use as standard coordinate system in our computations. Observe that $T^1_{p_S}S_a = T^1_{p_S}S$ and $T^1_{p_N}S_a = T^1_{p_N}S$.

The geodesic flow of $G_a$ is the Reeb flow of the Hilbert contact form $\alpha_a$ on $T^1S_a$. This contact form is defined by pulling back the canonical Liouville form of the cotangent bundle of $S$ by the Legendre transform $TS \rightarrow T^*S$ which is induced by the function $G_a^2/2$ and by restricting the resulting one-form to $T^1S_a$. Equivalently, the form $\alpha_a$ at a point $v\in T^1 S_a$ is given by
\[
g_v(v,d\pi(v)[ \cdot]),
\]
where $g_v$ denotes the second fiberwise differential of the function $G_a^2/2$ and $\pi: TS_a \rightarrow S_a$ is the footpoint projection. In the coordinates $(\theta,\beta,s)$ the contact form $\alpha_a$ has the expression
\begin{equation}\label{equation_contact_form_Zermelo}
\alpha_a = \frac{1}{1+ar(s)\cos \beta}( r(s) \cos \beta \,d\theta+\sin \beta \,ds).
\end{equation}
This can be checked as follows. The tangent bundle over the complement of the poles admits natural coordinates $(\theta,s,A,B) \simeq A\frac{\partial}{\partial \theta}+B\frac{\partial}{\partial s}$ induced by the coordinates $(\theta,s)$. If $(\theta,s,p_\theta,p_s) \simeq p_\theta \ d\theta \ + \ p_s \ ds$ denote the induced natural coordinates on the cotangent bundle over the complement of the poles, then the Legendre transform induced by $\frac{1}{2}G_a^2$ reads $p_\theta=G_a\frac{\partial G_a}{\partial A}$, $p_s=G_a\frac{\partial G_a}{\partial B}$. Using this and formula~\eqref{explicit_G_a}, the pull-back of the tautological $1$-form by Legendre transform yields $$ \frac{G_a}{1-r^2a^2} \left( \left( \frac{r^2A}{\sqrt{r^2A^2+(1-r^2a^2)B^2}} - r^2a \right) d\theta + \frac{(1-r^2a^2)B}{\sqrt{r^2A^2+(1-r^2a^2)B^2}} ds \right). $$ On the level $\{G_a=1\}$ we have, by~\eqref{explicit_G_a}, that $\sqrt{r^2A^2+(1-r^2a^2)B^2} = 1+r^2a(A-a)$. Plugging this together with $G_a=1$ above we obtain that this $1$-form agrees with
\[
\begin{aligned}
&\frac{1}{1-r^2a^2} \left( \frac{r^2A}{1+r^2a(A-a)} - r^2a \right) d\theta + \frac{B}{1+r^2a(A-a)} ds = \frac{r^2(A-a)d\theta+Bds}{1+r^2a(A-a)}
\end{aligned}
\]
on vectors tangent to the level $G_a^{-1}(1)$. In the coordinates induced by the diffeomorphism $\Psi_a$ we have  $A=a+\frac{1}{r}\cos\beta$, $B=\sin\beta$, or equivalently $r^2(A-a)=r\cos\beta$, $B=\sin\beta$. Plugging above we finally get the desired form of $\alpha_a$~\eqref{equation_contact_form_Zermelo} on $G_a^{-1}(1)$.


A direct computation shows that 
\begin{equation}\label{derivaalphaa}
\begin{split}
d\alpha_a = \frac{1}{(1+ar(s)\cos \beta)^2}\Bigl( & r(s)\sin \beta \, d\theta \wedge d\beta + (\cos \beta+ar(s))\, d\beta \wedge ds \\ & + r'(s)\cos \beta \, ds \wedge d\theta  \Bigr),
\end{split}
\end{equation}
and
\begin{equation}\label{volumeform}
\alpha_a \wedge d\alpha_a = \frac{r(s)}{(1+ar(s)\cos \beta)^2} \, d\theta \wedge d\beta \wedge ds.
\end{equation}
The Reeb vector field of $\alpha_a$ is
\[
R_a(\theta,\beta,s) = \frac{1}{r(s)}\Bigr( (\cos \beta +ar(s))\, \frac{\partial}{\partial \theta}  + r'(s)\cos \beta\, \frac{\partial}{\partial \beta}  + r(s)\sin \beta \, \frac{\partial}{\partial s}\Bigr),
\]
and thus its Reeb flow, corresponding to the geodesic flow of $G_a$, is determined by the system
\begin{equation}\label{sistema}
\left\{\begin{aligned}
  \dot \theta  = &  \frac{\cos \beta}{r(s)} + a\\
 \dot \beta  = & \frac{r'(s)\cos \beta}{r(s)}\\ 
  \dot s  =  & \sin \beta.
\end{aligned}\right.
\end{equation}
As in the Riemannian case ($a=0$), the Reeb flow preserves Clairault's integral 
\[
K(\beta,s)=r(s) \cos \beta.
\] 
In fact, the effect of the wind $W_a$ is only apparent in the first equation of \eqref{sistema}.

\section{The generating function in the Finsler case} 

Throughout this section, we assume without loss of generality that $a\geq 0$. Indeed, the geodesic flow on $S_a$ for $a<0$ is conjugate to the geodesic flow on $S_{-a}$.

As before we consider the Birkhoff annulus $A \subset T^1S_a$ associated to an equator $P_0 \subset S$ with least radius 
\begin{equation}\label{leastradius}
r_{\min} := \min \{r(s) \mid s\in (0,M/2), \; r'(s)=0\}.
\end{equation} 
This open annulus $A$ has natural coordinates $(\theta,\beta)\in \R/ 2\pi \Z \times (0,\pi)$. Observe that \eqref{cond_Finsler} implies that 
\[
0<a r_{\min}\leq a r_{\rm max} <1.
\]
We denote by $L=2 \pi r_{\min}$ the euclidean length of $P_0$. We shall systematically use the alternative coordinates $\xi=r_{\min} \theta\in \R/L \Z$ and $\eta=-\cos \beta\in (-1,1)$ on $A$.

Since the Reeb flow projected to the plane $(\beta,s)$ is independent of $a$, any Reeb trajectory of $\alpha_a$ starting at $A$ must return to $A$. This is proved in Lemma \ref{welldef} in the case $a=0$ and hence it must hold for any $a$. By the same reasoning, the first return time 
\[
\tau_a:A \to (0,+\infty),
\]
is independent of $a$. Hence we simply denote it by $\tau$. The rotational invariance of $G_a$ implies that $\tau$ does not depend on $\theta$.   

Let $\varphi_a: A \to A$ be the first return map.  Equations \eqref{sistema} and Clairault's integral give
\[
\varphi_a(\xi,\eta) = (\xi + f_a(\eta),\eta) ,
\]
for some smooth function $f_a:(-1,1)\to \R$ satisfying
\[
f_a = f +ar_{\min} \tau,
\]
where $f=f_0$ corresponds to the first return map in the Riemannian case $a=0$. 

Let $F:(-1,1)\rightarrow \R$ be the primitive of $f$ satisfying
\[
F(0)=\tau(0),
\]
and let $T:(-1,1)\rightarrow \R$ be the primitive of $\tau$ satisfying
\[
T(0)=0.
\]
Then
\begin{equation}\label{condFa2}
F_a:(-1,1)\to \R, \qquad F_a := F + a r_{\min} T,
\end{equation}
is the primitive of $f_a$ such that
\begin{equation}\label{condFa}
F_a(0)=\tau(0).
\end{equation}
Using the formula 
\[
\tau(\eta)=F(\eta)-\eta F'(\eta) \qquad \forall \eta \in (-1,1), 
\]
we  integrate $\tau$ by parts to obtain 
\begin{equation}\label{expT}
T(\eta) = \int_0^\eta \tau(\zeta)d\zeta =  \int_0^\eta \bigl( F(\zeta)-\zeta F'(\zeta) \bigr)\, d\zeta = 2\int_0^\eta F(\zeta) \, d\zeta -\eta F(\eta)\quad  \forall \eta \in (-1,1).
\end{equation}
Notice that the function $F$ is even, but the function $T$, and hence the function $F_a$, are not.
 Since $F$  extends continuously to $[-1,1]$, see Lemma \ref{lemma-area}, we conclude from \eqref{condFa2} and the expression above that the same is true for $F_a$.  By \eqref{condFa2}, \eqref{expT} and the evenness of $F$ we have
\begin{eqnarray}
\label{Fa-1} F_a(-1) & =& (1+ar_{\min})F(-1) -2 ar_{\min} \displaystyle{\int_0^1} F(\eta)d\eta, \\
\label{Fa1}F_a(1) & =& (1-ar_{\min})F(1) +2ar_{\min} \displaystyle{\int_0^1} F(\eta)d\eta.
\end{eqnarray}
Since
$T' = \tau >0$ we conclude from $T(0)=0$ and \eqref{condFa2}  that 
\begin{equation}\label{Fadesig}
\begin{aligned}
F_a & <F \ \ \  \mbox{ on } [-1,0),\\
F_a & > F \ \ \ \mbox{ on } (0,1].
\end{aligned}
\end{equation}
A critical point of $F_a$ corresponds to a fixed point of $\varphi_a$. The following result provides conditions for the existence of a fixed point of $\varphi_a$ with low first return time.

\begin{lem}\label{prop1}
If
\begin{equation}\label{intF}
\int_0^1 F(\eta)d\eta \leq \frac{L}{2}\left(1+ \frac{1}{(1+ar_{\min})^2} \right),
\end{equation} then $F$ admits a minimum point $\bar \eta\in (-1,0]$ satisfying 
\begin{equation}\label{critF}
F(\bar \eta) < \frac{L}{1+ar_{\min}},
\end{equation}
and $F_a$ admits a critical point $\hat \eta \in (-1,\bar \eta)$ satisfying
\begin{equation}\label{estimatau}
\tau(\hat \eta) < \tau(\bar \eta) = F(\bar \eta).
\end{equation}
\end{lem}

\begin{proof}
To prove the existence of a critical point $\bar \eta \in (-1,0]$ of $F$ satisfying inequality \eqref{critF} we argue indirectly and assume such a critical poinit does not exist. Using that $F(-1)=F(1)\geq L$  (see Lemma \ref{stima}) we thus have 
\[
F(\eta) \geq \frac{L}{1+ar_{\min}} \qquad \forall \eta \in [-1,1].
\]
Now using that $F$ is differentiable on $[0,1)$ and that  $F(\eta) \geq L\eta$ on $[0,1]$ (again by Lemma \ref{stima}), we obtain
\[
\begin{aligned}
\int_0^1 F(\eta)d\eta & > \int_0^1 \max \left \{ \frac{L}{1+ar_{\min}},L\eta\right \}d\eta\ =\ \int_0^{\frac{1}{1+ar_{\min}}} \frac{L}{1+ar_{\min}}d\eta + \int_{\frac{1}{1+ar_{\min}}}^1 L\eta d\eta \\ 
& = \frac{L}{(1+ar_{\min})^2} +\frac{L}{2}\left(1-\frac{1}{(1+ar_{\min})^2} \right) 
= \frac{L}{2}\left(1+\frac{1}{(1+ar_{\min})^2} \right),
\end{aligned}
\] 
contradicting \eqref{intF}. Since $F$ is even we conclude that $F$ admits a minimum point in $(-1,0]$ satisfying   \eqref{critF}.

Now we prove that $F_a$ admits a critical point $\hat \eta \in(-1,\bar \eta)$ satisfying \eqref{estimatau}.
Using  \eqref{Fa-1}, \eqref{intF} and  the fact that $F(-1)=F(1)\geq L$ we first observe that
\begin{equation}\label{valoralto}
F_a(-1)  \geq (1+ar_{\min})L - ar_{\min}L \left(1+ \frac{1}{(1+ar_{\min})^2} \right)
 =  L\left(1-  \frac{ar_{\min}}{(1+ar_{\min})^2}\right).
\end{equation} Moreover, by (\ref{Fadesig}) we have 
\begin{equation}\label{valorbaixo}
F_a(\bar \eta)  \leq F(\bar \eta) < \frac{L}{1+ar_{\min}},
\end{equation} where $\bar \eta \in (-1,0]$ is a minimum point of $F$ satisfying \eqref{critF}. Comparing \eqref{valoralto} and \eqref{valorbaixo} we see that 
\[
F_a(\bar \eta) < F_a(-1).
\] 
Now using that 
\begin{equation}
\label{pos}
F_a'(\bar \eta) = F'(\bar \eta) + ar_{\min} \tau(\bar \eta) = ar_{\min}\tau(\bar \eta)>0,
\end{equation}
we conclude that $F_a$ admits a critical point in $(-1,\bar \eta)$. At such a point the equality $F_a'=F' +ar_{\min}\tau=0$ holds. Let $\hat \eta$ be the largest critical point of $F_a$ in  $(-1,\bar \eta)$.  By (\ref{pos}), we have
\begin{equation}\label{Falinha}
F_a'=F'  +ar_{\min}\tau >0  \ \ \mbox{ on } (\hat \eta,\bar \eta]. 
\end{equation} 

The derivative of the function 
\[
g(\eta):= \frac{F(\eta)}{1-ar_{\min}\eta}, \quad \eta \in (-1,1),
\]
is
\[
g'(\eta) = \frac{F'(\eta) + a r_{\min}\bigl(F(\eta)-\eta F'(\eta)\bigr)}{ (1-ar_{\min}\eta)^2} = \frac{F'(\eta) + ar_{\min}\tau(\eta)}{(1-ar_{\min}\eta)^2}, 
\]
where we have used the expression  for $\tau$ which is given by Lemma \ref{genfun}. From \eqref{Falinha} we deduce that $g'$ is positive on $(\hat{\eta},\bar{\eta}]$, and hence $g$ is strictly increasing on this interval.

Since the derivative of $F_a$ at $\hat{\eta}$ vanishes, we have $F'(\hat \eta)=-ar_{\min}\tau(\hat \eta)$ and hence
\[
\tau(\hat{\eta}) = F(\hat{\eta}) - \hat{\eta} F'(\hat\eta) = F(\hat{\eta}) + a r_{\min} \hat\eta \tau(\hat\eta),
\]
which implies
\[
\tau(\hat \eta) = \frac{F(\hat \eta)}{1-ar_{\min}\hat \eta} = g(\hat{\eta}).
\]
As $g$ is strictly increasing on the interval $(\hat{\eta},\bar{\eta}]$ we obtain
\[
\tau(\hat \eta) = g(\hat{\eta}) < g(\bar{\eta}) = \frac{F(\bar \eta)}{1-ar_{\min}\bar \eta} \leq F(\bar \eta)=\tau(\bar \eta).
\] 
The last inequality follows from  $-1<\bar \eta \leq 0$. This proves \eqref{estimatau}.
\end{proof}

\section{Proof of Theorem \ref{mainthm2}}

We can now prove the second theorem stated in the introduction. 

\begin{mainthm}
Let $S\subset \R^3$ be a surface of revolution and let $a$ be a real number whose absolute value is smaller than $1/r_{\max}$, where $r_{\max}$ denotes the maximal distance of a point in $S$ from the $z$-axis. Then 
\[
\rho^{\rm HT}_{\rm sys}(S, G_a) \leq \rho^{\rm BH}_{\rm sys}(S, G_a) \leq \pi.
\]
The first inequality is an equality if and only if $a=0$. The second one is an equality if and only if $a=0$ and $S$ is Zoll.
\end{mainthm}

\begin{proof}
The first inequality follows from the inequality
\[
\mathrm{area}_{\rm HT}(S, G_a) \geq \mathrm{area}_{\rm BH}(S, G_a) = \mathrm{area} (S)
\]
which is discussed in the introduction. This inequality is an equality if and only if $a=0$.

Since we already know that for $a=0$ the systolic ratio of $S$ does not exceed $\pi$ and equals $\pi$ if and only if $S$ is Zoll, it is enough to prove the strict inequality
\begin{equation}\label{ell}
\ell_{\rm min}(S, G_a)^2 < \pi \, {\rm area}_{\rm BH}(S, G_a) = \pi \, {\rm area}(S) = \frac{{\rm vol}(T^1S)}{2}
\qquad \mbox{for all}\ \ \  a\neq 0.
\end{equation} 
As the geodesic flows on $S_a$ and $S_{-a}$ are conjugate, we can assume that $a>0$. When parametrized in the direction of $\partial/\partial \theta$, the equator of minimal radius $P_0$ is a  closed geodesic of $G_a$ with length
\[
\ell_0:=\frac{L}{1+ar_{\min}}\geq \ell_{\rm min}(S, G_a).
\]
We may assume that 
\[
\ell_0^2 \geq \frac{{\rm vol}(T^1S)}{2},
\] 
otherwise \eqref{ell} trivially holds.
The inequality \eqref{disvol} thus gives
\begin{equation}
\label{desigint1}
2\ell_0^2 = \frac{2L^2}{(1+ar_{\min})^2} \geq  {\rm vol}(T^1S) \geq 4L\int_0^1 F(\eta) d\eta -2L^2.
\end{equation}
In particular 
\[
\int_0^1 F(\eta)d\eta \leq \frac{L}{2} \left(1+\frac{1}{(1+ar_{\min})^2} \right).
\]
This inequality allows us to apply Lemma \ref{prop1} and gives us a minimum point $\bar \eta \in (-1,0]$ of $F$ such that 
\begin{equation}\label{con1}
\mu:=F(\bar \eta) < \frac{L}{1+ar_{\min}}<L.
\end{equation}
Moreover, this lemma also gives us a critical point $\hat \eta \in (-1,\bar \eta)$ of $F_a$ such that 
\begin{equation}\label{con2}
\tau(\hat \eta) <\mu.
\end{equation} Notice that $\hat \eta$ corresponds to a closed geodesic $\hat \gamma$ of $G_a$ whose length is 
\[
\ell(\hat \gamma) = \tau(\hat \eta).
\]
The existence of a minimum point $\bar \eta$ of $F$ satisfying  \eqref{con1}, together with the inequality $F(\eta)\geq L |\eta|$ from Lemma \ref{stima} and the differentiability of $F$, implies that 
\[
\int_0^1 F(\eta) d\eta >\int_0^1 \max\{\mu,L\eta\} \, d\eta =  \mu + \frac{1}{2} \frac{(L-\mu)^2}{L}.  
\] 
Hence \eqref{desigint1} and \eqref{con2} give
\[
{\rm vol}(T^1S) \geq 4L\left(\mu +  \frac{1}{2} \frac{(L-\mu)^2}{L} \right)-2L^2 = 2\mu^2  > 2\tau(\hat \eta)^2 = 2\ell(\hat \gamma)^2,
\] 
which implies
\[
\ell_{\min}(S,G_a)^2 < \frac{\mathrm{vol}(T^1S)}{2}
\]
proving \eqref{ell}. 
\end{proof}


\providecommand{\bysame}{\leavevmode\hbox to3em{\hrulefill}\thinspace}
\providecommand{\MR}{\relax\ifhmode\unskip\space\fi MR }
\providecommand{\MRhref}[2]{%
  \href{http://www.ams.org/mathscinet-getitem?mr=#1}{#2}
}
\providecommand{\href}[2]{#2}

\end{document}